\documentclass[a4paper]{article}
\usepackage{amsmath,amsthm,amssymb,amscd,fancybox}
\usepackage{mathrsfs}
\usepackage{amssymb}
\usepackage{ascmac}
\usepackage[arrow,matrix]{xy}
\usepackage{enumerate}
\usepackage{url}
\usepackage{graphicx}
\usepackage{stmaryrd}
\newtheorem{defn}{Definition}[section]
\newtheorem{lemma}[defn]{Lemma}
\newtheorem{thm}[defn]{Theorem}
\newtheorem{mainthm}[defn]{Main Theorem}
\newtheorem{prop}[defn]{Proposition}
\newtheorem{cor}[defn]{Corollary}

\newtheorem{rem}[defn]{Remark}

\def\rank{\mathop{\mathrm{rank}}\nolimits}

\def\Tr{\mathop{\mathrm{Tr}}\nolimits}

\def\Ker{\mathop{\mathrm{Ker}}\nolimits}
\def\Coker{\mathop{\mathrm{Coker}}\nolimits}

\def\Ann{\mathop{\mathrm{Ann}}\nolimits}

\def\Gal{\mathop{\mathrm{Gal}}\nolimits}

\def\Sel{\mathop{\mathrm{Sel}}\nolimits}

\def\nequiv{\equiv \hspace{-3.3mm}/\hspace{1.7mm}}

\makeatletter
   
   \@addtoreset{equation}{section}
\makeatother

\begin{document}
\title{On the plus and the minus Selmer groups for elliptic curves at supersingular primes}
\author{Takahiro KITAJIMA,\\
\it Department of Mathematics, Keio University\\
\it 3-14-1 Hiyoshi, Kohoku-ku, Yokohama, 223-8522, JAPAN\\
\it grenzwert@a6.keio.jp\\ \\
Rei OTSUKI\\
\it Department of Mathematics, Keio University\\
\it 3-14-1 Hiyoshi, Kohoku-ku, Yokohama, 223-8522, JAPAN\\
\it ray{\_}otsuki@math.keio.ac.jp
}
\date{}
\maketitle

\begin{abstract}
	Let $p$ be an odd prime number,
		$E$ an elliptic curve defined over a number field.
	Suppose that $E$ has good reduction 
		at any prime lying above $p$, and
		has supersingular reduction
		at some prime lying above $p$.
	In this paper,
		we construct the plus and the minus Selmer groups of $E$
		over the cyclotomic $\mathbb Z_p$-extension
		in a more general setting than that of B.D. Kim,
		and give a generalization of a result of B.D. Kim
		on the triviality of finite $\Lambda$-submodules
		of the Pontryagin duals of the plus and the minus Selmer groups,
		where $\Lambda$ is the Iwasawa algebra of the Galois group
		of the $\mathbb Z_p$-extension.
\end{abstract}
%\tableofcontents %章や節の目次を出力
%\listoffigures  %図の目次を出力
%\listoftables   %表の目次を出力

%%%%%%%%%%%%%%%%%%%%%%%%%%%%%%%%%%%%%%%%%%%%%%%%%%%%%%%%%%%%%%%%%%%%%%%%%%%%%%%%%%%%%%%%%%%%%%%%%%%%%%%%%%%%
%%%%%%%%%%%%%%%%%%%%%%%%%%%%%%%%%%%%%%%%%%%%%%%%%%%%%%%%%%%%%%%%%%%%%%%%%%%%%%%%%%%%%%%%%%%%%%%%%%%%%%%%%%%%
%%%%	§1 §1 §1 §1 §1 	%%%%%%%%	§1 §1 §1 §1 §1 	%%%%%%%%	§1 §1 §1 §1 §1 	%%%%
%%%%%%%%%%%%%%%%%%%%%%%%%%%%%%%%%%%%%%%%%%%%%%%%%%%%%%%%%%%%%%%%%%%%%%%%%%%%%%%%%%%%%%%%%%%%%%%%%%%%%%%%%%%%
%%%%%%%%%%%%%%%%%%%%%%%%%%%%%%%%%%%%%%%%%%%%%%%%%%%%%%%%%%%%%%%%%%%%%%%%%%%%%%%%%%%%%%%%%%%%%%%%%%%%%%%%%%%%
\section{Introduction}
	Let $p$ be an odd prime number,
		$F_0$ a finite extension of $\mathbb Q$,
		$F_{\infty}/F_0$ the cyclotomic $\mathbb Z_p$-extension
		and $F_n$ the $n$-th layer.
	Denote $\Lambda = \mathbb Z_p[[\Gal (F_{\infty}/F_0)]]$.
	Let $E$ be an elliptic curve defined over $F_0$.
	
	When $E$ has good ordinary reduction at any prime of $F_0$ lying above $p$,
		the Pontryagin dual of the $p$-primary Selmer group of $E$ over $F_{\infty}$
		is conjectured to be $\Lambda$-torsion.
	This conjecture is proved in several cases now.
	For example if the $p$-primary Selmer group of $E$ over $F_0$ is finite,
		or if $E$ is defined over $\mathbb Q$ and $F_0/\mathbb Q$ is abelian,
		then the conjecture is known to be true
		(cf. \cite{Gre99} p.55).
	
	On the contrary,
		when $E$ has good supersingular reduction at some prime of $F_0$ lying above $p$,
		the Pontryagin dual of the $p$-primary Selmer group of $E$ over $F_{\infty}$
		is no longer $\Lambda$-torsion.
	S. Kobayashi \cite{Kob03} defined the plus and the minus Selmer groups
		$\Sel ^{\pm}( F_{\infty}, E[p^{\infty}] )$
		when $E$ is defined over $\mathbb Q$, $a_p=1+p-\# \widetilde{E}(\mathbb F_p) =0$,
		and $F_0=\mathbb Q (\mu _p)$,
		where $\widetilde{E}$ denotes the reduction of $E$ at $p$,
		and $\mu _p$ denotes the group of $p$-th roots of unity.
	He proved that the Pontryagin duals $\Sel ^{\pm}( F_{\infty}, E[p^{\infty}] )^{\vee}$
		are $\Lambda$-torsion.
	A. Iovita and R. Pollack \cite{IP06} generalized
		definitions of the plus and the minus Selmer groups
		to the case when $F_0$ is a number field in which $p$ splits completely,
		$E$ is defined over $\mathbb Q$ and $a_p=0$.
	Further B.D. Kim \cite{Kim07} generalized them
		to the case when $F_0$ is a number field in which $p$ is unramified,
		$E$ is defined over $\mathbb Q$ and $a_p=0$.

	\bigskip
	
%	In this paper, we study finite $\Lambda$-submodules
%		of $\Sel^{\pm}( F_{\infty}, E[p^{\infty}] )^{\vee}$.
	
	B.D. Kim proved in \cite{Kim13}
		the following theorem on the triviality of finite $\Lambda$-submodules
		of $\Sel^{\pm} (F_{\infty},E[p^{\infty}])^{\vee}$.
%%%%%%%%%%%%%%%%%%%%%%%%%%%%%%%%%%%%%%%%%%%%%%%%%%%%%%%%%%%%%%%%%%%%%%%%%%%%
%%%%		[thm]	known results on the triviality of finite sub		%%%%
%%%%%%%%%%%%%%%%%%%%%%%%%%%%%%%%%%%%%%%%%%%%%%%%%%%%%%%%%%%%%%%%%%%%%%%%%%%%
	\begin{thm}[\cite{Kim13} Theorem 1.1]\label{known results on the triviality of finite sub}
		Let $F_0$ be a finite extension of $\mathbb Q$
			in which $p$ is unramified,
			$E$ an elliptic curve defined over $\mathbb Q$,
			and $a_p=0$.
		
		(1) Assume that $\Sel^{-} (F_{\infty},E[p^{\infty}])^{\vee}$ is $\Lambda$-torsion.
			Then $\Sel^- (F_{\infty},E[p^{\infty}])^{\vee}$
			has no nontrivial finite $\Lambda$-submodule.
		
		(2) Assume further that $p$ splits completely in $F_0$, and
			$\Sel^{+} (F_{\infty},E[p^{\infty}])^{\vee}$ is $\Lambda$-torsion.
			Then $\Sel^+ (F_{\infty},E[p^{\infty}])^{\vee}$
			has no nontrivial finite $\Lambda$-submodule.
	\end{thm}
	Throughout this paper,
		we assume that $\Sel^{\pm} (F_{\infty},E[p^{\infty}])^{\vee}$
		is $\Lambda$-torsion as in the above theorem.
	The following proposition
		on the local conditions of the plus and the minus Selmer groups
		was an important ingredient
		for Theorem \ref{known results on the triviality of finite sub}.
%%%%%%%%%%%%%%%%%%%%%%%%%%%%%%%%%%%%%%%%%%%%%%%%%%%%%%%%%%%%%%%%%%%%
%%%%		[prop]	known results on the local conditions		%%%%
%%%%%%%%%%%%%%%%%%%%%%%%%%%%%%%%%%%%%%%%%%%%%%%%%%%%%%%%%%%%%%%%%%%%
	\begin{prop}[\cite{Kim07} Proposition 3.13, \cite{Kim13} Proposition 2.2 and Proposition 2.3]\label{known results on the local conditions}
		Let $k_0$ be a finite unramified extension of $\mathbb Q_p$,
			and $k_{\infty}$ the cyclotomic $\mathbb Z_p$-extension of $k_0$.
		We denote the completed group ring $\mathbb Z_p[[\Gal (k_{\infty}/k_0)]]$ by $\Lambda$.
		
		(1) We have
		\begin{eqnarray*}
			(E^- (k_{\infty}) \otimes \mathbb Q_p/ \mathbb Z_p)^{\vee} \cong \Lambda ^{\oplus d}
		\end{eqnarray*}
		where $d=[k_0:\mathbb Q_p]$.
		
		(2) If $k_0= \mathbb Q_p$, we have
		\begin{eqnarray*}
			(E^+(k_{\infty}) \otimes \mathbb Q_p/\mathbb Z_p) ^{\vee} \cong \Lambda.
		\end{eqnarray*}
	\end{prop}
	When B.D. Kim studied the plus Selmer group in \cite{Kim13},
		he restricted the base field $F_0$
		as in Theorem \ref{known results on the triviality of finite sub} (2)
		to apply Proposition \ref{known results on the local conditions},
		and he got the result only in such a case.
	Even when he studied the minus Selmer group in \cite{Kim13},
		the base field $F_0$ was a finite extension of $\mathbb Q$
		in which $p$ is unramified.
	He did not consider the case when $\mu _p \subset F_0$.

	\bigskip

	In this paper, we consider more general fields for $F_0$ and $k_0$,
		and a more general elliptic curve $E$,
		than those of the above known results.
	We do not restrict these base fields even when we study the plus Selmer group.
	We also note that we also consider the case when $\mu_p \subset F_0$.
	We get the following result.
%%%%%%%%%%%%%%%%%%%%%%%%%%%%%%%%%%%%%%%%%%%%%%%%%
%%%%%		[mainthm]	main theorem II		%%%%%
%%%%%%%%%%%%%%%%%%%%%%%%%%%%%%%%%%%%%%%%%%%%%%%%%
	\begin{mainthm}[Theorem \ref{main theorem II}]\label{main theorem II (Intro)}
		Let $F$ be a finite extension of $\mathbb Q$,
			$F_0=F(\mu_p)$,
			$F_{\infty}/F_0$ the cyclotomic $\mathbb Z_p$-extension, and
			$E$ an elliptic curve defined over a subfield $F'$ of $F$.
		Let $S_p^{\rm ss}$ be the set of all primes of $F'$ lying above $p$
			where $E$ has supersingular reduction.
		Assume the following conditions:
		\begin{itemize}
			\item[\rm{(i)}] $E$ has good reduction at any prime of $F'$ lying above $p$,
			
			\item[\rm{(ii)}] $S_p^{\rm ss}$ is nonempty,
			
			\item[\rm{(iii)}] any prime $w \in S_p^{\rm ss}$ is unramified in $F$,
			
			\item[\rm{(iv)}] $F'_w=\mathbb Q_p$ for any prime $w \in S_p^{\rm ss}$,
				where $F'_w$ is the completion of $F'$ at the prime $w$,
			
			\item[\rm{(v)}] $a_w = 1 + p - \# \widetilde{E}_w (\mathbb F_p) =0$
				for any prime $w \in S_p^{\rm ss}$,
				where $\widetilde{E}_w$ is the reduction of $E$ at $w$, and
			
			\item[\rm{(vi)}] both $\Sel ^{\pm}(F_{\infty}, E[p^{\infty}])^{\vee}$
				are $\Lambda$-torsion.
		\end{itemize}
		Then both $\Sel ^{\pm}(F_{\infty}, E[p^{\infty}])^{\vee}$
			have no nontrivial finite $\Lambda$-submodule.
	\end{mainthm}
	
	\begin{rem}
		\rm{
		(1) We assume the condition (i)
				since we expect the condition (vi) automatically holds true
				under this condition.
		
		(2) In the case $S_p^{\rm ss} = \emptyset$,
				we have $\Sel ^{\pm}(F_{\infty},E[p^{\infty}]) = \Sel (F_{\infty}, E[p^{\infty}])$.
			Finite $\Lambda$-submodules of $\Sel (F_{\infty},E[p^{\infty}])^{\vee}$
				was studied by Hachimori and Matsuno \cite{Hachimori-Matsuno00}.
			Thus we will be interested in the case (ii).
		
		(3) The conditions (iii) and (iv) on the fields $F$ and $F'$
				will be used in applying the local result discussed in Section 3.
		
		(4) The condition (v) is crucial in this paper.
			In our method,
				it is important to study the local conditions $E^+ (k_n)$ and $E^- (k_n)$,
				where $k_n$ is the $n$-th layer of the cyclotomic $\mathbb Z_p$-extension $k_{\infty}/k_0$
				with $k_0$ a finite extension of $\mathbb Q_p$.
			In the case when $a_w \neq 0$ for some $w \in S_p^{\rm ss}$,
				one might need another submodule of $\Sel (F_{\infty}, E[p^{\infty}])$
				instead of $\Sel ^{\pm}(F_{\infty}, E[p^{\infty}])$.
			F. Sprung \cite{Sprung12} defined $\Sel^{\sharp}(F_{\infty}, E[p^{\infty}])$
				and $\Sel^{\flat}(F_{\infty}, E[p^{\infty}])$
				instead of the plus and the minus Selmer groups
				in the case when $F_0=\mathbb Q (\mu _p)$,
				and $E$ is defined over $\mathbb Q$
				which has supersingular reduction at $p$
				without assuming $a_p = 0$.
			He defined $E^{\sharp} (k_{\infty})$ and $E^{\flat} (k_{\infty})$
				in the case when $k_0= \mathbb Q_p (\mu _p)$,
				however, did not define $E^{\sharp} (k_n)$ nor $E^{\flat} (k_n)$.
			We can not yet apply our method
				in the case when $a_w \neq 0$ for some $w$.
		
		(5) In some cases, $\Sel ^{\pm}(F_{\infty}, E[p^{\infty}])^{\vee}$
				is actually known to be $\Lambda$-torsion.
			For example, let $F_0$ be a finite abelian extension of $\mathbb Q$,
				and $F_{\infty}/F_0$ the cyclotomic $\mathbb Z_p$-extension.
			Suppose that $E$ is defined over $\mathbb Q$
				and has supersingular reduction at $p$ with $a_p =0$.
			In this case,
				one can actually show that $\Sel ^{\pm} (F_{\infty}, E[p^{\infty}])^{\vee}$ is $\Lambda$-torsion.
			Our main theorem implies that
				$\Sel ^{\pm} (F_{\infty}, E[p^{\infty}])^{\vee}$ has no nontrivial finite $\Lambda$-submodule
				for any finite abelian field $F_0$.
			On the other hand, we need a certain assumption on $F_0$
				to apply B.D. Kim's result.
		}
	\end{rem}
	
	For the proof of Theorem \ref{main theorem II (Intro)},
		we will generalize Proposition \ref{known results on the local conditions}
		to the case of our setting.
	In the study of $(E^{\pm}(k_{\infty}) \otimes \mathbb Q_p/\mathbb Z_p)^{\vee}$,
		we find that its $\Lambda$-module structure in our setting is generally different from
		those in the settings of previous works.
	We now explain some known results on the $\Lambda$-module structure of
		$(E^{\pm}(k_{\infty}) \otimes \mathbb Q_p/\mathbb Z_p)^{\vee}$.
	
	Takeji \cite{Takeji14} considered the case when
		$k_0$ is a quadratic unramified extension of $\mathbb Q_p$.
	He generalized Proposition \ref{known results on the local conditions} to this case,
		i.e. he proved that $(E^{\pm}(k_{\infty}) \otimes \mathbb Q_p/\mathbb Z_p)^{\vee}$
		is a free $\Lambda$-module of $\Lambda$-rank $2$.
	Applying this result,
		he proved that $\Sel ^{\pm}(F_{\infty},E[p^{\infty}])^{\vee}$
		has no nontrivial finite $\Lambda$-submodule
		in the case when $F_0$ is a quadratic number field in which $p$ inerts,
		which is a generalization of Theorem \ref{known results on the triviality of finite sub}.
%	He defined a system of local points,
%		a key tool as in \cite{Kim07}, by using Lubin-Tate theory.
%	His construction of the system of local points forced him to assume
%		that $F_{0,v}$ (the completion of $F_0$ at the place $v$) is
%		at most quadratic unramified extension of $\mathbb Q_p$
%		for all primes $v$ of $F_0$ lying above $p$.
	
	M. Kim proved in his dissertation \cite{MKim11}
		that $E^{\pm}(k_n)$ are cyclic $\mathbb Z_p[\Gal (k_n/\mathbb Q_p)]$-modules for all $n$,
		where $k_n$ is the $n$-th layer of the cyclotomic $\mathbb Z_p$-extension $k_{\infty}/k_0$,
		in the case when $k$ is a general finite unramified extension of $\mathbb Q_p$ and $k_0=k(\mu _p)$,
		however, he did not notice that
		the assumption $[k:\mathbb Q_p] \nequiv 0$ (mod $4$) is needed.
	From this cyclicity, one can show that
		$(E^{\pm}(k_{\infty}) \otimes \mathbb Q_p/\mathbb Z_p)^{\vee}$
		is a free $\Lambda$-module of $\Lambda$-rank $[k_0:\mathbb Q_p]$,
		which is a generalization of Proposition \ref{known results on the local conditions}.
	
	B.D. Kim \cite{Kim14} independently generalized Proposition \ref{known results on the local conditions}
		to the case when $k_0$ is a finite unramified extension of $\mathbb Q_p$
		and $[k_0:\mathbb Q_p] \nequiv 0$ (mod $4$),
		i.e. he proved in this case that
		$(E^{\pm}(k_{\infty}) \otimes \mathbb Q_p/\mathbb Z_p)^{\vee}$
		is a free $\Lambda$-module of $\Lambda$-rank $[k_0:\mathbb Q_p]$
		(cf. \cite{Kim14} Theorem 2.8).
	Applying this result,
		we can generalize Theorem \ref{known results on the triviality of finite sub} to the case
		when $F_0$ is a finite extension of $\mathbb Q$ in which $p$ is unramified
		and $[F_{0,v}:\mathbb Q_p] \nequiv 0$ (mod $4$) for all primes $v$ of $F_0$ lying above $p$,
		where $F_{0,v}$ is the completion of $F_0$ at the prime $v$.
		
	We remark that we consider more general settings
		than those of all the above known results.

	\bigskip

%	In this paper, we generalize this B.D. Kim's result for more general $F_0$,
%		assuming both $\Sel ^{\pm}(F_{\infty},E[p^{\infty}])^{\vee}$ are $\Lambda$-torsion,
%		as in the following theorem.
	
	Here we prepare some notations of our settings
		and explain an obstruction for generalizing Proposition \ref{known results on the local conditions}
		to the case of our setting, which we have overcome in this paper.
	Let $k$ be a finite unramified extension of $\mathbb Q_p$ of degree $d$,
		$k_0=k(\mu _{p})$,
		$k_{\infty}$ the cyclotomic $\mathbb Z_p$-extension of $k_0$,
		$k_n$ the $n$-th layer,
		$\Delta = \Gal (k(\mu _p)/k)$,
		$\Gamma = \Gal (k_{\infty}/k_0)$,
		$\Gamma _n = \Gal (k_{\infty}/k_n)$,
		$G_n = \Gal (k_n/\mathbb Q_p)$
		and $\Lambda = \mathbb Z_p[[\Gamma]]$.
	
	An essential property,
		expected in all previous works \cite{Kim07}, \cite{Kim13}, \cite{Kim14}, \cite{MKim11}, and \cite{Takeji14}
		was that $E^{\pm} (k_n)$ are cyclic $\mathbb Z_p[G_n]$-modules for all $n$.
	From this expected property, we can show that
		$(E^{\pm}(k_{\infty}) \otimes \mathbb Q_p/\mathbb Z_p)^{\vee}$
		is a free $\Lambda$-module of $\Lambda$-rank $[k_0:\mathbb Q_p]$.
	On the contrary to this expectation, we find that
		$E^+ (k_n)$ are not cyclic $\mathbb Z_p[G_n]$-modules
		when $d \equiv 0$ (mod $4$)
		(cf. Proposition \ref{generators of norm subgroups}
		and Remark \ref{remark on d_{-1}}).
	An obstruction is that
		this non-cyclicity makes $(E^+(k_{\infty}) \otimes \mathbb Q_p/\mathbb Z_p)^{\vee}$ more complicated.
	In fact, we find that
		$(E^+(k_{\infty}) \otimes \mathbb Q_p/\mathbb Z_p)^{\vee}$
		is not a free $\Lambda$-module
		in the case when $d \equiv 0$ (mod $4$)
		(cf. Remark \ref{Remark on Z_p-ranks of Gamma_n coinvariants}).
	Therefore,
		the same statement with the conclusion of Proposition \ref{known results on the local conditions}
		does not hold in the general setting.
	
	A crucial step for the proof of our main theorem is
		to investigate the $\Lambda$-module structure of
		$(E^{\pm}(k_{\infty}) \otimes \mathbb Q_p/\mathbb Z_p )^{\vee}$
		more precisely.
	We prove the following proposition on the local conditions
		of the plus and the minus Selmer groups,
		which is a generalization of Proposition \ref{known results on the local conditions}
		and an important ingredient for our main theorem.
	
%%%%%%%%%%%%%%%%%%%%%%%%%%%%%%%%%%%%%%%%%%%%%%%%%%%%
%%%%		[prop]	key proposition (Intro)		%%%%
%%%%%%%%%%%%%%%%%%%%%%%%%%%%%%%%%%%%%%%%%%%%%%%%%%%%
	\begin{prop}[Proposition \ref{structure of the plus and the minus local conditions}]
		\label{key proposition (Intro)}
		$(E^{\pm}(k_{\infty})
			\otimes \mathbb Q_p/\mathbb Z_p)^{\vee}$
			has no nontrivial finite $\Lambda$-submodule
			and its $\Lambda$-rank is $[k_0:\mathbb Q_p]$.
	\end{prop}
	
	We prove this proposition by calculating the $\Gamma _n$-coinvariants
		of the $\chi$-component $(E^{\pm}(k_{\infty})^{\chi} \otimes \mathbb Q_p/\mathbb Z_p)^{\vee}$
		for all $n$,
		where $\chi:\Delta \rightarrow \mathbb Z_p^{\times}$ is a character of $\Delta$.
	The original idea for such calculations,
		to calculate the $\Gamma$-coinvariants and to apply Nakayama's lemma,
		was suggested by M. Kurihara.
	The authors are very grateful to him.
	
	As a consequence of Proposition \ref{key proposition (Intro)},
		we prove the following proposition which is also an important ingredient
		for our main theorem.
%%%%%%%%%%%%%%%%%%%%%%%%%%%%%%%%%%%%%%%%%%%%%%%%%%%%%%%%%%%%%%%%%%%%%%%%%%%%
%%%%		[prop]	Lambda module structure of H^1/E^{pm} (Intro)		%%%%
%%%%%%%%%%%%%%%%%%%%%%%%%%%%%%%%%%%%%%%%%%%%%%%%%%%%%%%%%%%%%%%%%%%%%%%%%%%%
	\begin{prop}[Proposition \ref{Lambda module structure of H^1/E^{pm}}]\label{key proposition 2 (Intro)}
		We have
		\begin{eqnarray*}
			\left(
			\frac{H^1(k_{\infty},E[p^{\infty}])}
			{E^{\pm}(k_{\infty}) \otimes \mathbb Q_p/\mathbb Z_p}
			\right) ^{\vee} \cong \Lambda ^{\oplus [k_0:\mathbb Q_p]}.
		\end{eqnarray*}
	\end{prop}
	
	\bigskip
	
	In comparison with the strategy for the proof of Theorem \ref{known results on the triviality of finite sub},
		we develop another strategy also for the proof of our main theorem.
	Our strategy is to prove
		that the triviality of finite $\Lambda$-submodules
		of $\Sel(F_{\infty},E[p^{\infty}])^{\vee}$ is inherited to
		that of $\Sel^{\pm}(F_{\infty},E[p^{\infty}])^{\vee}$,
		and use the following theorem.
%%%%%%%%%%%%%%%%%%%%%%%%%%%%%%%%%%%%%%%%%%%%%%%%%%%%
%%%%		[thm]	main theorem I (Intro)		%%%%
%%%%%%%%%%%%%%%%%%%%%%%%%%%%%%%%%%%%%%%%%%%%%%%%%%%%
	\begin{thm}[Theorem \ref{main theorem I}]
		Assume that both $\Sel ^{\pm}(F_{\infty},E[p^{\infty}])^{\vee}$
			are $\Lambda$-torsion.
		Then $\Sel (F_{\infty}, E[p^{\infty}])^{\vee}$
			has no nontrivial finite $\Lambda$-submodule.
	\end{thm}
	
	\bigskip
	
%%%%%%%%%%%%%%%%%%%%%%%%%%%%%%%%%%%%%%%%%%%%%%%%%%%%%%%%%%%%%%%%%%%%
%%%%		[rem]	key proposition of B.D. Kim and M. Kim		%%%%
%%%%%%%%%%%%%%%%%%%%%%%%%%%%%%%%%%%%%%%%%%%%%%%%%%%%%%%%%%%%%%%%%%%%
%	\begin{rem}[cf. \cite{Kim07} Proposition 3.13, \cite{Kim13} Proposition 2.2 and 2.3, \cite{Kim14} Proposition 2.8]
%		\label{key proposition in B.D. Kim and M.Kim}
%		\rm{B.D. Kim (and M. Kim essentially) determined the explicit structure
%			of the $\Lambda$-module $(E^{\pm}(k_{\infty})^{\Delta} \otimes \mathbb Q_p/\mathbb Z_p)^{\vee}$.
%		Indeed, they proved that
%		\begin{eqnarray*}
%			(E^{\pm} (k_{\infty})^{\Delta} \otimes \mathbb Q_p/ \mathbb Z_p)^{\vee}
%			\cong \Lambda ^{\oplus d}
%		\end{eqnarray*}
%		unless $\pm = +$ and $d \equiv 0$ (mod $4$).
%		
%		In the proof of our main theorem,
%			we use Proposition \ref{key proposition (Intro)}
%			instead of the explicit structure of $(E^{\pm}(k_{\infty}) \otimes \mathbb Q_p/\mathbb Z_p)^{\vee}$.}
%	\end{rem}
	
	The above discussion is enough to prove our main theorem,
		however,
		we can determine the explicit structure of the $\Lambda$-module
		$(E^{\pm}(k_{\infty})^{\chi} \otimes \mathbb Q_p/\mathbb Z_p)^{\vee}$,
		on which we explain here.
	We prepare some notations concerning the character decomposition.
	Let $\chi:\Delta \rightarrow \mathbb Z_p^{\times}$ be a character
		of $\Delta = \Gal (k (\mu _p)/k)$.
	If $M$ is a $\mathbb Z_p[\Delta]$-module,
		then $M$ is decomposed into
		\begin{eqnarray*}
			M = \bigoplus _{\chi}\varepsilon _{\chi} M,
		\end{eqnarray*}
		where $\varepsilon _{\chi}
		= \frac{1}{p-1} \sum _{\sigma \in \Delta} \chi (\sigma) \sigma ^{-1}
		\in \mathbb Z_p[\Delta]$.
	We denote by $M^{\chi}$ the $\chi$-component $\varepsilon_{\chi}M$.
	We fix a topological generator $\gamma \in \Gamma$,
		and identify $\mathbb Z_p[[\Gamma]]$
		with the ring of formal power series $\mathbb Z_p[[X]]$
		by identifying $\gamma$ with $1+X$.
	We get the following theorem.
	
%%%%%%%%%%%%%%%%%%%%%%%%%%%%%%%%%%%%%%%%%%%%%%%%%%%%%%%%%%%%
%%%%		[thm]	supplementary theorem (Intro)		%%%%
%%%%%%%%%%%%%%%%%%%%%%%%%%%%%%%%%%%%%%%%%%%%%%%%%%%%%%%%%%%%
	\begin{thm}[Theorem \ref{supplementary theorem}]\label{supplementary theorem (Intro)}
		Let $\chi :\Delta \rightarrow \mathbb Z_p^{\times}$ be a character.
		We have
		\begin{eqnarray*}
			\left( E^+(k_{\infty})^{\chi}
				\otimes \mathbb Q_p/\mathbb Z_p \right) ^{\vee}
				& \cong &
				\Lambda ^{\oplus d} \oplus
				(\Lambda /X)^{\oplus \delta},\\
			\left( E^-(k_{\infty})^{\chi}
				\otimes \mathbb Q_p/\mathbb Z_p \right) ^{\vee}
				& \cong &
				\Lambda ^{\oplus d}
		\end{eqnarray*}
		where
		\begin{eqnarray*}
			\delta =
			\left\{
			\begin{array}{ll}
				0 & \text{ if } d \nequiv 0\ (\text{mod } 4)
					\text{ or } \chi \neq \mathbf 1,\\[2mm]
				2 & \text{ otherwise}.
			\end{array}
			\right.
		\end{eqnarray*}
	\end{thm}

	\bigskip
	
	The outline of this paper is as follows.
	In Section 2,
		we define the plus and the minus Selmer groups following Kobayashi,
		and fix a global setting.
	In Section 3,
		we study the local conditions in a local setting.
	Subsection 3.1 is a preparation for the rest of Section 3.
	In Subsection 3.2,
		we give a description of $E^{\pm} (k_n)^{\chi}$
		in terms of formal groups and a system of local points.
	In Subsection 3.3,
		we study the $\Lambda$-modules $(E^{\pm}(k_{\infty}) ^{\chi} \otimes \mathbb Q_p/\mathbb Z_p)^{\vee}$
		and $(H^1(k_{\infty},E[p^{\infty}])/(E^{\pm}(k_{\infty}) \otimes \mathbb Q_p/\mathbb Z_p))^{\vee}$.
	In Subsection 3.4,
		we further determine the explicit structure of the $\Lambda$-module
		$(E^{\pm}(k_{\infty}) ^{\chi} \otimes \mathbb Q_p/\mathbb Z_p)^{\vee}$.
	In Section 4, we study finite $\Lambda$-submodules of the Pontryagin duals of the Selmer groups
		$\Sel(F_{\infty},E[p^{\infty}])^{\vee}$ and $\Sel^{\pm}(F_{\infty},E[p^{\infty}])^{\vee}$.
	In Subsection 4.1, we prove
		that the usual $p$-primary Selmer group has no nontrivial finite $\Lambda$-submodule
		under the same assumption with the main theorem.
	In this step, it is essential to assume that
		both $\Sel ^{\pm}(F_{\infty}, E[p^{\infty}])^{\vee}$ are $\Lambda$-torsion.
	In Subsection 4.2, we prove our main theorem.

%%%%%%%%%%%%%%%%%%%%%%%%%%%%%%%%%%%%%%%%%%%%%%%%%%%%%%%%%%%%%%%%%%%%%%%%%%%%%%%%%%%%%%%%%%%%%%%%%%%%%%%%%%%%
%%%%%%%%%%%%%%%%%%%%%%%%%%%%%%%%%%%%%%%%%%%%%%%%%%%%%%%%%%%%%%%%%%%%%%%%%%%%%%%%%%%%%%%%%%%%%%%%%%%%%%%%%%%%
%%%%	§2 §2 §2 §2 §2 	%%%%%%%%	§2 §2 §2 §2 §2 	%%%%%%%%	§2 §2 §2 §2 §2 	%%%%
%%%%%%%%%%%%%%%%%%%%%%%%%%%%%%%%%%%%%%%%%%%%%%%%%%%%%%%%%%%%%%%%%%%%%%%%%%%%%%%%%%%%%%%%%%%%%%%%%%%%%%%%%%%%
%%%%%%%%%%%%%%%%%%%%%%%%%%%%%%%%%%%%%%%%%%%%%%%%%%%%%%%%%%%%%%%%%%%%%%%%%%%%%%%%%%%%%%%%%%%%%%%%%%%%%%%%%%%%
\section{The plus and the minus Selmer groups}
	
	Let $p$ be a prime number,
		$F$ a finite extension of $\mathbb Q$,
		and $E$ an elliptic curve defined over $F$.
	For a finite extension $K/F$,
		the $p$-primary Selmer group for $E$ over $K$
		is defined by
		\begin{eqnarray*}
			&&\Sel (K,E[p^{\infty}])
			:= \Ker \left(
				H^1(K,E[p^{\infty}]) \longrightarrow
				\prod _{v} \frac{H^1(K_{v}, E[p^{\infty}])}
				{E(K_{v}) \otimes \mathbb Q_p/\mathbb Z_p}\right),
		\end{eqnarray*}
		where $v$ runs through all places of $K$,
		$K_{v}$ is the completion of $K$ at the place $v$,
%		$E[p^{\infty}]$ is the set of all $p$-power torsion points of $E(\overline{\mathbb Q})$,
		and $E(K_{v}) \otimes \mathbb Q_p/\mathbb Z_p$ is regarded
		as a subgroup of $H^1(K_{v}, E[p^{\infty}])$ by the Kummer map.
	For a number field $K$ that is an infinite extension of $F$,
		we define the $p$-primary Selmer group for $E$ over $K$ by
		$$
			\Sel (K, E[p^{\infty}])
			:= \varinjlim_{K'} \Sel (K',E[p^{\infty}]),
		$$
		where $K'$ runs through all the subfields of $K$
		which are finite extensions of $F$,
		and transition maps are restriction maps
		between cohomology groups.
	
	We denote $F_n = F(\mu _{p^{n+1}})$, $F_{-1}=F$ and $F_{\infty}= \bigcup _n F_n$,
		where $\mu_{p^n}$ denotes the group of $p^n$-th roots of unity.
		We fix a generator $(\zeta_{p^n})$ of $\mathbb Z_p(1)$,
		namely, for each $n \geq 0$, $\zeta_{p^n}$ is a primitive $p^n$-th root of unity
		such that $\zeta_{p^{n+1}}^p=\zeta_{p^n}$.
	
	Then by definition,
		we have
		$$
			\Sel (F_{\infty},E[p^{\infty}])
			= \underset{\substack{\longrightarrow \\ n}}{\lim}\ \Sel (F_n, E[p^{\infty}]).
		$$

	Throughout this paper,
		we fix the following notations:
		\begin{itemize}
		\item
		$p$ is an odd prime number,
			
		\item
		$F$ is a finite extension of $\mathbb Q$,
		
		\item
		$E$ is an elliptic curve defined over a subfield $F'$ of $F$.
		\end{itemize}
	
	Denote $S_p^{\rm ss}$ the set of all primes of $F'$ lying above $p$
		where $E$ has supersingular reduction.
	
	Throughout this paper,
		we assume the following:
		\begin{itemize}
		\item
		$E$ has good reduction at any prime $w|p$ of $F'$,
		
		\item
		$S_p^{\rm ss}$ is nonempty,
		
		\item
		any prime $w \in S_p^{\rm ss}$ is unramified in $F$,
		
		\item
		$F'_w=\mathbb Q_p$ for any prime $w \in S_p^{\rm ss}$,
			where $F'_w$ is the completion of $F'$ at the prime $w$, and
		
		\item
		$a_w=1+p -\# \widetilde{E}_w (\mathbb F_p)=0$
			for any prime $w \in S_p^{\rm ss}$,
			where $\widetilde{E}_w$ is the reduction of $E$ at $w$.
		\end{itemize}
	
	When $p \geq 5$, the condition $a_w=0$ is automatically satisfied
		since we have $p|a_w$ and $|a_w| \leq 2 \sqrt{p}$.\\
	
	Denote $S_{p,F}^{\rm ss}$ the set of all primes of $F$ lying above $S_p^{\rm ss}$.
	
	Following S. Kobayashi \cite{Kob03}
		we define subgroups $E^+(F_{n,v})$ and $E^-(F_{n,v})$ of $E(F_{n,v})$
		for each prime $v \in S_{p,F}^{\rm ss}$,
		and define plus and minus Selmer groups $\Sel^{\pm}(F_n, E[p^{\infty}])$,
		$\Sel^{\pm}(F_{\infty},E[p^{\infty}])$ as the following
		(see also \cite{Kim07} and \cite{MKim11}).
	
%%%%%%%%%%%%%%%%%%%%%%%%%%%%%%%%%%%%%%%%%%%%%%%%%%%%%%%%
%%%%		[defn]	plus and minus subgroups		%%%%
%%%%%%%%%%%%%%%%%%%%%%%%%%%%%%%%%%%%%%%%%%%%%%%%%%%%%%%%
	\begin{defn}\label{plus and minus subgroups}
		(1)
		For a prime $v \in S_{p,F}^{\rm ss}$ and $n \geq -1$,
			let $F_{n,v}$ be the completion of $F_n$ at the unique prime of $F_n$
			lying above $v$.
		We define
		\begin{eqnarray*}
			&&E^+ (F_{n,v}) = \{ P \in E(F_{n,v})
				| \Tr_{n/{m+1}} P \in E(F_{m,v})
				\text{ for all even } m, -1 \leq m \leq n-1 \}, \\
			&&E^- (F_{n,v}) = \{ P \in E(F_{n,v})
				| \Tr_{n/{m+1}} P \in E(F_{m,v})
				\text{ for all odd } m, -1 \leq m \leq n-1 \}
		\end{eqnarray*}
		where $\Tr _{n/m+1}: E(F_{n,v}) \rightarrow E(F_{m+1,v})$ is the trace map.
		
		(2) The plus and the minus Selmer groups are defined by
		\begin{eqnarray*}
			&&\Sel ^{\pm}(F_n, E[p^{\infty}])
			:= \Ker \left( \Sel (F_n, E[p^{\infty}]) \longrightarrow
				\bigoplus _{v \in S_{p,F}^{\rm ss}}
				\frac{H^1(F_{n,v},E[p^{\infty}])}
				{E^{\pm}(F_{n,v}) \otimes \mathbb Q_p/\mathbb Z_p} \right),\\
			&&\Sel ^{\pm}(F_{\infty}, E[p^{\infty}])
			:= \underset{\substack{\longrightarrow \\ n}}{\lim}\ \Sel^{\pm} (F_n, E[p^{\infty}]).
		\end{eqnarray*}
	\end{defn}
	
	We denote the Pontryagin dual of a module $M$ by $M^{\vee}$.
	Let $\mathcal G_n = \Gal (F_n/F)$ and $\mathcal G_{\infty}= \Gal (F_{\infty}/F)$.
	Then $\mathbb Z_p[\mathcal G_n]$ acts naturally
		on $\Sel^{\pm} (F_n,E[p^{\infty}])^{\vee}$
		and $\Lambda ({\mathcal G_{\infty}}) := \mathbb Z_p[[\mathcal G_{\infty}]]$
		on $\Sel^{\pm} (F_{\infty}, E[p^{\infty}])^{\vee}$.
	
	The Pontryagin dual of the usual $p$-primary Selmer group $\Sel (F_{\infty},E[p^{\infty}])^{\vee}$
		is not a torsion $\Lambda ({\mathcal G_{\infty}})$-module, however,
		$\Sel^{\pm} (F_{\infty}, E[p^{\infty}])^{\vee}$ is known to be
		$\Lambda ({\mathcal G_{\infty}})$-torsion
		in the case $F= \mathbb Q$ (cf. \cite{Kob03} Theorem 2.2).
	
%%%%%%%%%%%%%%%%%%%%%%%%%%%%%%%%%%%%%%%%%%%%%%%%%%%%%%%%%%%%%%%%%%%%%%%%%%%%%%%%%%%%%%%%%%%%%%%%%%%%%%%%%%%%
%%%%%%%%%%%%%%%%%%%%%%%%%%%%%%%%%%%%%%%%%%%%%%%%%%%%%%%%%%%%%%%%%%%%%%%%%%%%%%%%%%%%%%%%%%%%%%%%%%%%%%%%%%%%
%%%%	§3 §3 §3 §3 §3 	%%%%%%%%	§3 §3 §3 §3 §3 	%%%%%%%%	§3 §3 §3 §3 §3 	%%%%
%%%%%%%%%%%%%%%%%%%%%%%%%%%%%%%%%%%%%%%%%%%%%%%%%%%%%%%%%%%%%%%%%%%%%%%%%%%%%%%%%%%%%%%%%%%%%%%%%%%%%%%%%%%%
%%%%%%%%%%%%%%%%%%%%%%%%%%%%%%%%%%%%%%%%%%%%%%%%%%%%%%%%%%%%%%%%%%%%%%%%%%%%%%%%%%%%%%%%%%%%%%%%%%%%%%%%%%%%
\section{The formal groups and the norm subgroups}
	
	Let $E/\mathbb Q_p$ be an elliptic curve with $a_p=0$ and
		$\widehat{E}$ the formal group over $\mathbb Z_p$
		associated with the minimal model of $E$ over $\mathbb Q_p$.
	Let $k$ be a finite unramified extension of $\mathbb Q_p$
		of degree $d=[k:\mathbb Q_p]$
		and $\mathcal O_k$ the ring of integers of $k$.
	For each $n \geq -1$,
		let $k_n = k(\mu _{p^{n+1}})$
		and $\mathfrak m_n$ be the maximal ideal of $k_n$.
	Let $k_{\infty} = \bigcup _{n \geq -1} k_n$
		and $\mathfrak m_{\infty} = \bigcup _{n \geq -1} \mathfrak m_n$.
	Let $G_n = \Gal (k_n/\mathbb Q_p)$,
		$G_{\infty} = \Gal (k_{\infty}/\mathbb Q_p)$,
		$\Gamma _n = \Gal (k_{\infty}/k_n)$,
		$\Gamma = \Gamma _0 (= \Gal (k_{\infty}/k_0))$
		and $\Delta = \Gal (k(\mu_p)/k)=\Gal (k_0/k_{-1})$.
	Let $\varphi$ be the Frobenius homomorphism in $\Gal (k/\mathbb Q_p) = G_{-1}$
	characterized by $x^{\varphi} \equiv x^p$ (mod $p\mathcal O_k$).
	We denote $\Lambda = \mathbb Z_p[[\Gamma]]$.
	We fix a topological generator $\gamma \in \Gamma$.
	Then we identify $\mathbb Z_p[[\Gamma]]$
		with $\mathbb Z_p[[X]]$,
		and $\mathbb Z_p[[G_{\infty}]]$ with $\mathbb Z_p[G_0][[X]]$
		by identifying $\gamma$ with $1+X$.
	
%%%%%%%%%%%%%%%%%%%%%%%%%%%%%%%%%%%%%%%%%%%%%%%%%%%%%%%%%%%%%%%%
%%%%%%%%%%%%%%%%%%%%%%%%%%%%%%%%%%%%%%%%%%%%%%%%%%%%%%%%%%%%%%%%
%%%%	§3.1 §3.1 §3.1 	%%%%%%%%	§3.1 §3.1 §3.1	%%%%
%%%%%%%%%%%%%%%%%%%%%%%%%%%%%%%%%%%%%%%%%%%%%%%%%%%%%%%%%%%%%%%%
%%%%%%%%%%%%%%%%%%%%%%%%%%%%%%%%%%%%%%%%%%%%%%%%%%%%%%%%%%%%%%%%
	\subsection{The formal groups associated to $E$}
	
%%%%%%%%%%%%%%%%%%%%%%%%%%%%%%%%%%%%%%%%%%%%%%%%
%%%%		[prop]	E^ is torsion-free		%%%%
%%%%%%%%%%%%%%%%%%%%%%%%%%%%%%%%%%%%%%%%%%%%%%%%
	\begin{prop}\label{E^ is torsion-free}
		For any $n$, $\widehat{E}(\mathfrak m_n)$ is $\mathbb Z_p$-torsion-free.
	\end{prop}
	
	\begin{proof}
		We can prove this by the same method as the proof of \cite{Kob03} Proposition 8.7.
	\end{proof}
	
	The above proposition implies that
		the formal logarithm $\log _{\widehat{E}} (X)$
%		$\log _{\widehat{E}}:\widehat{E} \rightarrow \widehat{\mathbb G}_a$
		induces an injective homomorphism
		$\log _{\widehat{E}}:\widehat{E}(\mathfrak m_n) \rightarrow k_n$
		for all $n$,
		since the kernel of the logarithm of a formal group
		precisely consists of the elements of finite order.
		
	For such a one-dimensional formal group $\mathscr F$
		defined over $\mathbb Z_p$ with height 2,
		the formal logarithm $\log _{\mathscr F}$
		induces isomorphisms as in the following proposition
		(cf. The proof of Proposition 2.1 in \cite{Kurihara02},
		and Lemma 2.4 in \cite{Hazewinkel77}).
%%%%%%%%%%%%%%%%%%%%%%%%%%%%%%%%%%%%%%%%
%%%%		[prop]	log isoms 1		%%%%
%%%%%%%%%%%%%%%%%%%%%%%%%%%%%%%%%%%%%%%%
	\begin{prop}\label{log isoms 1}
		Let $\mathscr F$ be a one-dimensional formal group defined over $\mathbb Z_p$
			with height 2.
		For a finite extension $K/\mathbb Q_p$,
			denote by $\mathfrak m_K$ its maximal ideal.
		Then the logarithm
			$\log _{\mathscr F} : \mathscr F (\mathfrak m_K) \rightarrow K$
			induces isomorphisms
		\begin{eqnarray*}
			\log _{\mathscr F}:\mathscr F(\mathfrak m_K^j)
			\overset{\simeq}{\longrightarrow}
			\mathfrak m_K^j
		\end{eqnarray*}
		for all $j > v_K(p)/(p^2-1)$,
			where $v_K$ is the normalized valuation of $K$
			so that $v_K(\pi _K)=1$ for a uniformizer $\pi _K$ of $K$.
	\end{prop}

	Following \cite{Kim07} and \cite{MKim11} we construct
		a system of local points $(d_n)_{n}$.
	
	Fix a generator $\zeta$ of the group of roots of unity in $k$.
	Then $\zeta$ is a primitive $(p^d-1)$-th root of unity,
		and we have $k = \mathbb Q_p (\zeta)$.
	
	Let $g(X) = (X+ \zeta)^p-\zeta ^p \in \mathcal O_k[X]$,
		$g^{(m)}(X) = g^{\varphi ^{m-1}} \circ g^{\varphi ^{m-2}} \circ \cdots \circ g(X)
		=(X+\zeta)^{p^m}-\zeta ^{p^m}$
		for $m \geq 1$ and $g^{(0)}(X)=X$.
%	This Eisenstein polynomial $g(X)$ of degree $p$ satisfies that
%	\begin{eqnarray*}
%		g(X) \equiv X^p & (\text{mod } p), \\[1mm]
%		g(X) \equiv \varpi X & (\text{mod } \deg 2), \\[1mm]
%		\text{coefficient of } X^{p-1} = \zeta p,&
%	\end{eqnarray*}
%	where $\varpi := \zeta ^{p-1}p$.
%	This number $\varpi \in k$ satisfies $N_{k/\mathbb Q_p}(\zeta ^{p-1}p)=1$.		
%	Let $g^{(n)}(X) = g^{\varphi ^n} \circ g^{\varphi ^{n-1}} \circ \cdots \circ g(X)$
%		for $n \geq 0$ and $g^{(-1)}(X) = X$.
%	From the Lubin-Tate group theory we can see that any root $\pi _{n+1}$ of $g^{(n)}(X)$
%		that is not a root of $g^{(n-1)}(X)$ is a uniformizer of $\mathfrak m_n$
%		and also satisfies $k_n=k(\pi _{n+1})$.
%	Each $\pi _{n+1}$ can be explicitly written as $\zeta^{\varphi ^{-(n+1)}} (\zeta_{p^{n+1}}^i-1)$ for some $i$
%		relatively prime to $p$.
	We define a formal power series $\log _{\mathscr G}(X)$ by
	$$
		\log _{\mathscr G}(X) = \sum _{m=0}^{\infty} (-1)^m \frac{g^{(2m)}(X)}{p^m}.
	$$
	We can check that
	\begin{eqnarray*}
		(\log _{\mathscr G}^{\varphi ^{-(n+1)}})^{\varphi ^2}(X^{p^2})
		+ p\log _{\mathscr G}^{\varphi ^{-(n+1)}}(X) \equiv 0\ (\text{mod } p)
	\end{eqnarray*}
	and $(\log _{\mathscr G}^{\varphi ^{-(n+1)}})'(X) \in \mathcal O_k[[X]]$ for each $n$.
	This means that $\log _{\mathscr G}^{\varphi ^{-(n+1)}}(X)$
		is of the Honda type $t^2+p$ for each $n$.
	Hence by Honda theory (cf. \cite{Hon70}) we see that
	\begin{itemize}
	\item there is a formal group $\mathscr G_n$
		defined over $\mathcal O_k$
		whose formal logarithm $\log _{\mathscr G_n}$
		is given by $\log _{\mathscr G}^{\varphi ^{-(n+1)}}$
		for each $n$, and
	\item the power series $\exp _{\widehat{E}} \circ \log _{\mathscr G_n}$
		is contained in $\mathcal O_k[[X]]$ and gives an isomorphism
		$\mathscr G_n \rightarrow \widehat{E}$ over $\mathcal O_k$ for each $n$.
	\end{itemize}
	
	We fix a generator $(\zeta _{p^n})$ of $\mathbb Z_p(1)$,
		namely, for each $n \geq 0$,
		$\zeta _{p^n}$ is a primitive $p^n$-th root of unity
		such that $\zeta _{p^{n+1}}^p = \zeta _{p^n}$.
	Let $\pi _n = \zeta ^{\varphi ^{-(n+1)}} (\zeta _{p^{n+1}}-1)
		\in \mathfrak m_n$ for $n \geq -1$
		and $\pi _n =0$ for $n < -1$.
	For each $n$, we can easily show that
%%%%%%%%%%%%%%%%%%%%%%%%%%%%%%%%%%%%%%%%%%%%%%%%
%%%%		a system of uniformizers		%%%%
%%%%%%%%%%%%%%%%%%%%%%%%%%%%%%%%%%%%%%%%%%%%%%%%
	\begin{eqnarray}\label{a system of uniformizers}
	g^{(m),\varphi ^{-(n+1)}} (\pi _n) = \pi _{n-m}
	\end{eqnarray}
	for any $m \geq 0$ by direct calculation.
	
	Put
	\begin{eqnarray*}
		\varepsilon _n &=& \zeta ^{\varphi ^{-(n+3)}}p
			- \zeta ^{\varphi ^{-(n+5)}}p^2
			+ \zeta ^{\varphi ^{-(n+7)}}p^3 - \cdots \\
		&=& \sum _{i=1}^{\infty} (-1)^{i-1}
			\zeta ^{\varphi ^{-(n+1+2i)}}p^i \in \mathfrak m_k
	\end{eqnarray*}
	for $n \geq -1$.
	Since $\log _{\mathscr G_n} : \mathscr G_n(\mathfrak m_k) \rightarrow \mathfrak m_k$
		is an isomorphism  for all $n$ (cf. Proposition \ref{log isoms 1}),
		there is $\epsilon _n \in \mathscr G_n(\mathfrak m_k)$
		such that $\log _{\mathscr G_n}(\epsilon _n) = \varepsilon _n$
		foe $n \geq -1$.

%%%%%%%%%%%%%%%%%%%%%%%%%%%%%%%%%%%%%%%%%%%%%%%%%%%%%%%%%%%%%%%%%%%%%%%%%%%%%%%%%%%%
%%%%		[defn]	a system of local points following B.D. Kim and M. Kim		%%%%
%%%%%%%%%%%%%%%%%%%%%%%%%%%%%%%%%%%%%%%%%%%%%%%%%%%%%%%%%%%%%%%%%%%%%%%%%%%%%%%%%%%%
	\begin{defn}\label{a system of local points following B.D. Kim and M. Kim}
		We define
		$$
			d_n = \exp _{\widehat{E}} \circ \log _{\mathscr G_n}
			(\epsilon _n [+]_{\mathscr G_n} \pi_n)
		$$
		for $n \geq -1$,
			where $[+]_{\mathscr G_n}$ is the addition of $\mathscr G_n$.
	\end{defn}

	For $n \geq m$, we denote by $\Tr_{n/m}:\widehat{E}(\mathfrak m_n)
	\rightarrow \widehat{E}(\mathfrak m_m)$ the trace (norm)
	with respect to the group-law $\widehat{E}(X,Y)$.

%%%%%%%%%%%%%%%%%%%%%%%%%%%%%%%%%%%%%%%%%%%%
%%%%		[prop]	trace condition		%%%%
%%%%%%%%%%%%%%%%%%%%%%%%%%%%%%%%%%%%%%%%%%%%
	\begin{prop}\label{trace condition}
		The system of local points
		$( d_n ) _n \in \prod _{n \geq -1} \widehat{E}(\mathfrak m_n)$
		satisfies
		
		(1) $\Tr _{n/n-1}(d_n)=-d_{n-2}$ for each $n \geq 1$,
		
		(2) $\Tr _{0/-1}(d_0)=-(\varphi + \varphi ^{-1}) d_{-1}$.
	\end{prop}
	
	\begin{proof}
		We prove this by the same method as the proof of Lemma 8.9 in \cite{Kob03}.
		Since $\log _{\widehat{E}}$ is injective
			(cf. Proposition \ref{E^ is torsion-free})
			and commute with the action of $G_n$ on $\widehat{E}(\mathfrak m_n)$,
			it is enough to show that the relation holds
			after applying $\log _{\widehat{E}}$
			to both sides of the equality.
		
		We have
		\begin{eqnarray*}
			\log_{\widehat{E}} (d_n)
			&=& \log _{\mathscr G_n} (\epsilon _n
				[+]_{\mathscr G_n} \pi _n)\\
			&=& \log_{\mathscr G_n}(\epsilon _n)
				+ \log_{\mathscr G_n}(\pi _n)\\
			&=& \varepsilon _n
				+ \sum _{m=0}^{[\frac{n+1}{2}]} (-1)^m \frac{\pi _{n-2m}}{p^m}.
		\end{eqnarray*}
		Here the last equality follows from (\ref{a system of uniformizers})
			and $\pi _n = 0$ for $n \leq -1$.
		
		For $n \geq 1$, we have
		\begin{eqnarray*}
			\Tr_{n/n-1} \log _{\widehat{E}}(d_n)
			&=& p \varepsilon _n
				- \zeta ^{\varphi ^{-(n+1)}} p
				+ \sum _{m=1}^{[\frac{n+1}{2}]}(-1)^m \frac{\pi _{n-2m}}{p^{m-1}}\\
			&=& - \varepsilon _{n-2} - \sum _{m=0}^{[\frac{n-1}{2}]}
				(-1)^m \frac{\pi _{n-2-2m}}{p^m}\\
			&=& - \log _{\widehat{E}} (d_{n-2}).
		\end{eqnarray*}
		
		For $n=0$, we have
		\begin{eqnarray*}
			\Tr _{0/-1} \log _{\widehat{E}}(d_0)
			&=& (p-1)\varepsilon _0 - \zeta ^{\varphi ^{-1}}p \\
			&=& -(\varphi + \varphi ^{-1}) \varepsilon_{-1} \\
			&=& -(\varphi + \varphi ^{-1}) \log _{\widehat{E}} (d_{-1}).
		\end{eqnarray*}
	\end{proof}

%%%%%%%%%%%%%%%%%%%%%%%%%%%%%%%%%%%%%%%%%%%%%%%%%%%%%%%%%%%%%%%%%%%%
%%%%		[rem]	varphi + varphi ^{-1} always appears		%%%%
%%%%%%%%%%%%%%%%%%%%%%%%%%%%%%%%%%%%%%%%%%%%%%%%%%%%%%%%%%%%%%%%%%%%
	\begin{rem}\label{varphi + varphi ^{-1} always appears}
		\rm{As long as we define local points as values
			of certain power series at certain points,
			the factor $\varphi + \varphi ^{-1}$
			in the condition (2) always appears
			(cf. \cite{Kob13} Proposition 3.10, (3.3)).
		Although B.D. Kim did not mention explicitly in \cite{Kim07}, \cite{Kim13},
			this factor $\varphi + \varphi^{-1}$ was an obstruction.
		He assumed in \cite{Kim07} and \cite{Kim13}
			that $k=\mathbb Q_p$
			when he considered the plus Selmer groups
			in order to make the situation simpler, i.e. $\varphi + \varphi ^{-1}=2$
			in $\mathbb Z_p[G_{-1}] = \mathbb Z_p$.
		In this paper, we consider general unramified extension $k/\mathbb Q_p$,
			carefully taking into account this factor $\varphi + \varphi ^{-1}$
			in $\mathbb Z_p[G_{-1}]$.}
	\end{rem}

%%%%%%%%%%%%%%%%%%%%%%%%%%%%%%%%%%%%%%%%%%%%%%%%%%%%%%%%%%%%
%%%%		[lemma]	on varphi + varphi^{-1} and d		%%%%
%%%%%%%%%%%%%%%%%%%%%%%%%%%%%%%%%%%%%%%%%%%%%%%%%%%%%%%%%%%%
	\begin{lemma}\label{on varphi + varphi^{-1} and d}
		$\varphi + \varphi ^{-1}$ is a unit in $\mathbb Z_p[G_{-1}]$
		if and only if $d \nequiv 0$ (mod $4$).
	\end{lemma}
	
	\begin{proof}
		First we note that $\varphi + \varphi ^{-1} \in \mathbb Z_p[G_{-1}]^{\times}$
		if and only if $1+ \varphi ^2 \in \mathbb Z_p[G_{-1}]^{\times}$.

		If $d$ is odd, we have
		\begin{eqnarray*}
			(1+\varphi ^2)(1 - \varphi ^2+ \varphi ^4
				- \cdots + (-1)^{\frac{2d-2}{2}} \varphi ^{2d-2})
			&=& 1+(-1)^{d-1} \\
			&=& 2.
		\end{eqnarray*}
		
		If $d$ is even, we have
		\begin{eqnarray*}
			(1+ \varphi ^2)(1 - \varphi ^2 + \varphi ^4
				- \cdots + (-1)^{\frac{d-2}{2}} \varphi ^{d-2})
			&=& 1 + (-1)^{\frac{d-2}{2}}\\
			&=& \left\{
			\begin{array}{ll}
				\displaystyle
				2  & \text{if } d \nequiv 0\ (\text{mod }4), \\[3mm]
				\displaystyle
				0  & \text{if } d \equiv 0\ (\text{mod }4),
			\end{array}\right.
		\end{eqnarray*}
		and $1 - \varphi ^2 + \varphi ^4 - \cdots + (-1)^{\frac{d-2}{2}}\varphi ^{d-2}
			\neq 0$ in $\mathbb Z_p[G_{-1}]$.
		Since $2$ is invertible in $\mathbb Z_p[G_{-1}]$,
			we get the conclusion of Lemma \ref{on varphi + varphi^{-1} and d}.
	\end{proof}

%%%%%%%%%%%%%%%%%%%%%%%%%%%%%%%%%%%%%%%%%%%%%%%%%%%%%%%%%%%%%%%%%%%%%%%%
%%%%		[rem]	on the lemma on varphi + varphi^{-1} and d		%%%%
%%%%%%%%%%%%%%%%%%%%%%%%%%%%%%%%%%%%%%%%%%%%%%%%%%%%%%%%%%%%%%%%%%%%%%%%
	\begin{rem}\label{on the lemma on varphi + varphi^{-1} and d}
		\rm{In the proof of Lemma \ref{on varphi + varphi^{-1} and d},
			we have proved the following;
			(1) $\varphi + \varphi ^{-1}$ is a unit if $d \nequiv 0$ (mod $4$),
			(2) $\varphi + \varphi ^{-1}$ is a zero-divisor if $d \equiv 0$ (mod $4$).
		}
	\end{rem}
	
	Moreover, we can easily check the following lemma.
%%%%%%%%%%%%%%%%%%%%%%%%%%%%%%%%%%%%%%%%%%%%%%%%%%%%%%%%%%%%%%%%%%%%%%%%
%%%%		[lemma]	on the annihilator of varphi + varphi^{-1}		%%%%
%%%%%%%%%%%%%%%%%%%%%%%%%%%%%%%%%%%%%%%%%%%%%%%%%%%%%%%%%%%%%%%%%%%%%%%%
	\begin{lemma}\label{on the annihilator of varphi + varphi^{-1}}
		We have
		\begin{eqnarray*}
			&& \Ann _{\mathbb Z_p[G_{-1}]}(\varphi + \varphi ^{-1})\\[3mm]
			&& = \left\{
			\begin{array}{ll}
				0 & \text{if } d \nequiv 0\ (\text{mod }4),\\[3mm]
				\langle 1- \varphi ^2 + \varphi ^4 - \cdots - \varphi ^{d-2}
				\rangle _{\mathbb Z_p[G_{-1}]} & \text{if } d \equiv 0\ (\text{mod }4).
			\end{array}
			\right.
		\end{eqnarray*}
		and $\rank _{\mathbb Z_p}\Ann _{\mathbb Z_p[G_{-1}]}(\varphi + \varphi ^{-1})=2$
			if $d \equiv 0$ (mod $4$).
	\end{lemma}

	To describe the quotient modules
		$\widehat{E}(\mathfrak m_n)/\widehat{E}(\mathfrak m_{n-1})$
		using the local points $(d_n)_n$,
		we prepare the following lemma.
	
%%%%%%%%%%%%%%%%%%%%%%%%%%%%%%%%%%%%%%%%%%%%
%%%%		[lemma]	on m_n/m_{n-1}		%%%%
%%%%%%%%%%%%%%%%%%%%%%%%%%%%%%%%%%%%%%%%%%%%
	\begin{lemma}\label{on m_n/m_{n-1}}
		We have $\mathfrak m_n/\mathfrak m_{n-1}
			= \langle \pi _n \rangle _{\mathbb Z_p[G_n]}$
			for $n \geq 0$.
	\end{lemma}
	
	\begin{proof}
		It is enough to show that $\zeta (\zeta _{p^{n+1}}-1)$ generates
			$\mathfrak m_n/\mathfrak m_{n-1}$ as a $\mathbb Z_p[G_n]$-module.
		
		We first observe the ring of integers $\mathcal O_{k_n}$ of $k_n$.
		Let $P_m = \{ (\zeta _{p^m} -1)^{\tau} | \tau \in \Gal (k_m/k)\}$ for $m \geq 1$
			and $P_0 =\{ 1\}$.
%		Then we see that
%			$\mathcal O_{k_n} = 
%			\langle P_0 \cup \cdots \cup P_{n+1}\rangle _{\mathcal O_k}$.
		Since $\mathcal O_{k_n}= \mathcal O_k[\zeta _{p^{n+1}}]$,
			we have
		$$
			\mathcal O_{k_n} = \langle
				P_0 \cup P_1 \cup \cdots \cup P_{n+1}
				\rangle _{\mathcal O_k}.
		$$
		Thus, for $x \in \mathcal O_{k_n}$,
			we can write
			$x = a_0 + \sum _{m=1}^{n+1} \sum _{\tau \in \Gal (k_m/k)}
			a_{m,\tau} (\zeta _{p^m}-1)^{\tau}$
			with $a_0, a_{m,\tau} \in \mathcal O_k$.
		With this notation,
			since each $(\zeta _{p^m}-1)^{\tau}$ already has positive valuation,
			we see that $x \in \mathfrak m_n$
			if and only if $a_0 \in \mathfrak m_k$.
		
		Take any class in $\mathfrak m_n/\mathfrak m_{n-1}$
			with a representative $x \in \mathfrak m_n$.
		Write $x = a_0 + \sum _{m=1}^{n+1} \sum_{\tau \in \Gal (k_m/k)}
			a_{m,\tau}(\zeta _{p^{n+1}} -1)^{\tau}$.
		In this summation,
			the summands $a_0$ and $a_{m,\tau}(\zeta _{p^m}-1)^{\tau}$
			with $1 \leq m \leq n$
			are contained in $\mathfrak m_{n-1}$.
		Since $\mathcal O_k = \mathbb Z_p[\zeta]
			=\langle \zeta \rangle _{\mathbb Z_p[G_{-1}]}$,
			each $a_{n+1,\tau} \in \mathcal O_k$ can be written as
			$a_{n+1,\tau}=\sum _{i=0}^{d-1} b_{\tau ,i} \zeta^{\varphi ^i}$
			with $b _{\tau ,i} \in \mathbb Z_p$.
		Therefore we have
		\begin{eqnarray*}
			x & \equiv & \sum _{\tau \in \Gal (k_{n+1}/k)}
				a_{n+1,\tau} (\zeta _{p^{n+1}} - 1)^{\tau} \\
			& \equiv & \sum _{\tau \in \Gal (k_{n+1}/k)}
				\sum _{i=0}^{d-1} b_{\tau ,i} \zeta ^{\varphi ^i}
				(\zeta _{p^{n+1}} - 1)^{\tau}\quad
				(\text{mod } \mathfrak m_{n-1}).
		\end{eqnarray*}
		Here, $\zeta ^{\varphi ^i} (\zeta _{p^{n+1}}-1)^{\tau}$
			for each $i$ with $0 \leq i \leq d-1$,
			and each $\tau \in \Gal (k_{n+1}/k)$,
			is exactly a Galois conjugate of $\zeta (\zeta _{p^{n+1}}-1)$
			by $G_n= G_{-1} \times \Gal (k_{n+1}/k)$.
		This completes the proof.
	\end{proof}

%%%%%%%%%%%%%%%%%%%%%%%%%%%%%%%%%%%%%%%%
%%%%		[prop]	log isoms 2		%%%%
%%%%%%%%%%%%%%%%%%%%%%%%%%%%%%%%%%%%%%%%
	\begin{prop}\label{log isoms 2}
		For $n \geq 0$,
			we have $\log _{\widehat{E}} (\widehat{E}(\mathfrak m_n))
			\subseteq \mathfrak m_n + k_{n-1}$
			and the formal logarithm $\log _{\widehat{E}}$ induces
			canonical isomorphisms of $\mathbb Z_p[G_n]$-modules,
		\begin{eqnarray*}
			\widehat{E}(\mathfrak m _n)/\widehat{E}(\mathfrak m_{n-1})
			\overset{\simeq}{\longrightarrow} \mathfrak m_n/\mathfrak m_{n-1}.
		\end{eqnarray*}
		By these isomorphisms, $d_n$ is sent to $\pi_n$.
		In particular, we have
		\begin{eqnarray*}
			\widehat{E}(\mathfrak m_n) /\widehat{E}(\mathfrak m_{n-1})
			= \langle d_n\rangle _{\mathbb Z_p[G_n]}. 
		\end{eqnarray*}
	\end{prop}
	
	\begin{proof}
		We prove this by the same method as the proof of Proposition 8.11 in \cite{Kob03}
			or Proposition 4.9 in \cite{IP06}.
		
		For the first statement,
			we only note that by the commutative diagram
			\[
			\xymatrix{
				\widehat{E}(\mathfrak m_n) \ar[r]^{\log _{\widehat{E}}}
				\ar[d]^{\cong}_{\exp _{\mathscr G_{-1}} \circ \log _{\widehat{E}}} & k_n\\
				\mathscr G_{-1}(\mathfrak m_n) \ar[ru]_{\log _{\mathscr G_{-1}}}&
			},
			\]
			it is enough to consider $\log _{\mathscr G_{-1}}(= \log _{\mathscr G})$
			on $\mathscr G_{-1}(\mathfrak m_n)$
			instead of $\log _{\widehat{E}}$ on $\widehat{E}(\mathfrak m_n)$.
		Then we can show that
			$\log _{\mathscr G_{-1}}(\mathscr G_{-1} (\mathfrak m_n))
			\subseteq \mathfrak m_n + k_{n-1}$
			as in \cite{Kob03}, \cite{IP06}.
		
		Since we have $\log _{\widehat{E}}(\mathfrak m_n) \cap k_{n-1}
			= \log _{\widehat{E}}(\mathfrak m_{n-1})$,
			the natural map
			$$
				\widehat{E} (\mathfrak m_n)
				/\widehat{E} (\mathfrak m_{n-1})
				\longrightarrow (\mathfrak m_n+k_{n-1})/k_{n-1}
				\cong \mathfrak m_n/ \mathfrak m_{n-1}
			$$
			is injective.
		By direct calculation, we have
			$\log _{\widehat {E}} (d_n) \equiv \pi _n$ (mod $k_{n-1}$).
		Since $\pi _n$ generates $\mathfrak m_n/\mathfrak m_{n-1}$
			as a $\mathbb Z_p[G_n]$-module (cf. Lemma \ref{on m_n/m_{n-1}}),
			the above injection is in fact a bijection.
	\end{proof}

%%%%%%%%%%%%%%%%%%%%%%%%%%%%%%%%%%%%%%%%%%%%%%%%%%%%
%%%%		[cor]	generators of E^(m_n)		%%%%
%%%%%%%%%%%%%%%%%%%%%%%%%%%%%%%%%%%%%%%%%%%%%%%%%%%%
	\begin{cor}\label{generators of E^(m_n)}
		We have
		\begin{eqnarray*}
			\widehat{E}(\mathfrak m_n)
			=\left\{
			\begin{array}{ll}
				\displaystyle
				\langle d_{-1} \rangle _{\mathbb Z_p[G_{-1}]}
					& \quad \text{if } n=-1, \\[3mm]
				\displaystyle
				\langle d_n, d_{n-1} \rangle _{\mathbb Z_p[G_n]}
					& \quad \text{if } n \geq 0.
			\end{array}\right.
		\end{eqnarray*}
	\end{cor}
	
	\begin{proof}
		The case $n=-1$ follows from $\widehat{E} (\mathfrak m_{-1}) \cong \mathfrak m_{-1}$
			(see Proposition \ref{log isoms 1}) and Nakayama's lemma.
		The case $n \geq 0$ follows easily from Proposition \ref{log isoms 2}
			and the trace relations satisfied by the $d_n$
			(see Proposition \ref{trace condition}).
	\end{proof}
	
%%%%%%%%%%%%%%%%%%%%%%%%%%%%%%%%%%%%%%%%%%%%%%%%%%%%%%%%%%%%%%%%%%%%%%%%
%%%%		[rem]	requisitions on a system of local points		%%%%
%%%%%%%%%%%%%%%%%%%%%%%%%%%%%%%%%%%%%%%%%%%%%%%%%%%%%%%%%%%%%%%%%%%%%%%%
	\begin{rem}\label{requisitions on a system of local points}
		\rm{We defined the system of local points $(d_n)_n$
			following B.D. Kim \cite{Kim07} and M. Kim \cite{MKim11} in the above.
		We can take another system of local points instead of $(d_n)_n$.
		Indeed, what we need for the following discussion
			is a system of local points $(d_n)_n$
			which satisfies the following three conditions
		\begin{enumerate}
			\item $\Tr _{n/n-1} (d_n) = -d_{n-2}$
				for each $n \geq 1$ (Proposition \ref{trace condition} (1)),
			\item $\Tr _{0/-1}(d_0) = - (\varphi + \varphi ^{-1}) d_{-1}$
				(Proposition \ref{trace condition} (2)),
			\item $\widehat{E}(\mathfrak m_n)/\widehat{E}(\mathfrak m_{n-1})
				= \langle d_n \rangle _{\mathbb Z_p[G_n]}$
				(Proposition \ref{log isoms 2}).
		\end{enumerate}
		Such a system $(d_n)_n$ obviously admits at least a difference
			of multiplication by a unit in $\mathbb Z_p[G_{-1}]^{\times}$.
		S. Kobayashi constructed such a system of local points also
			in \cite{Kob13} Proof of Proposition 3.12
			by using another formal power series
			$\ell _{\epsilon}(X)$
			and a system $(\zeta _{p^{n+1}}-1)_n$
			instead of $\log _{\mathscr G} (X)$
			and a system $(\epsilon _n [+]_{\mathscr G_n} \pi _n)_n$.
		In our setting,
			the formal power series $\ell _{\epsilon}(X)$ is defined
			for each $\epsilon \in \widehat{E}(\mathfrak m_k)$ by
			\begin{eqnarray*}
				\ell _{\epsilon}(X)
				= \varepsilon + \sum _{m=0}^{\infty} (-1)^m
				\frac{f^{(2m)}(\varepsilon ' X)}{p^m} \in k[[X]]
			\end{eqnarray*}
			where $\varepsilon = \log _{\widehat{E}}(\epsilon) \in \mathfrak m_k$,
			$\varepsilon ' =(\varphi ^2 + p) \varepsilon p^{-1} \in \mathcal O_k$,
			$f(X) = (X+1)^p-1$ and
			$f^{(m)}(X)$ is the $m$-iterated composition of $f$.
		By using this formal power series,
			Kobayashi defined $d_{\epsilon ,n}$ for each $n \geq -1$ by
			\begin{eqnarray*}
				d_{\epsilon ,n}
				= \exp _{\widehat{E}} \circ \ell _{\epsilon}^{\varphi ^{-(n+1)}} (\zeta _{p^{n+1}}-1)
				\in \widehat{E} (\mathfrak m_n).
			\end{eqnarray*}
		Then the first and the second conditions, which are listed above, are satisfied.
		If we take $\varepsilon \in \mathfrak m_k$
			such that $\mathfrak m_k = \langle \varepsilon \rangle _{\mathbb Z_p[G_n]}$,
			then the third condition is also satisfied.
		We also note that we can take such $\varepsilon \in \mathfrak m_k$,
			since $\mathfrak m_k$ is known to be a cyclic $\mathbb Z_p[G_n]$-module.
		}
	\end{rem}

%%%%%%%%%%%%%%%%%%%%%%%%%%%%%%%%%%%%%%%%%%%%%%%%%%%%%%%%%%%%%%%%
%%%%%%%%%%%%%%%%%%%%%%%%%%%%%%%%%%%%%%%%%%%%%%%%%%%%%%%%%%%%%%%%
%%%%	§3.2 §3.2 §3.2 	%%%%%%%%	§3.2 §3.2 §3.2	%%%%
%%%%%%%%%%%%%%%%%%%%%%%%%%%%%%%%%%%%%%%%%%%%%%%%%%%%%%%%%%%%%%%%
%%%%%%%%%%%%%%%%%%%%%%%%%%%%%%%%%%%%%%%%%%%%%%%%%%%%%%%%%%%%%%%%
\subsection{The norm subgroups}

	Following S. Kobayashi \cite{Kob03} (and M. Kim \cite{MKim11}),
			we define the $n$-th plus subgroup $\widehat{E}^+(\mathfrak m_n)$,
			the $n$-th minus subgroup $\widehat{E}^-(\mathfrak m_n)$
			and the $n$-th norm subgroup $\mathscr C(\mathfrak m_n)$
			of $\widehat{E}(\mathfrak m_n)$;
			
%%%%%%%%%%%%%%%%%%%%%%%%%%%%%%%%%%%%%%%%%%%%
%%%%		[defn]	norm subgroups		%%%%
%%%%%%%%%%%%%%%%%%%%%%%%%%%%%%%%%%%%%%%%%%%%
	\begin{defn}\label{norm subgroups}
		We define
		\begin{eqnarray*}
			&&\widehat{E}^+(\mathfrak m_n)
			= \{ P \in \widehat{E}(\mathfrak m_n)
				|\Tr_{n/m+1} P \in \widehat{E}(\mathfrak m_m)\ \text{\rm for all even }
				m, -1 \leq m \leq n-1\},\\
			&&\widehat{E}^-(\mathfrak m_n)
			= \{ P \in \widehat{E}(\mathfrak m_n)
				|\Tr_{n/m+1} P \in \widehat{E}(\mathfrak m_m)\ \text{\rm for all odd }
				m, -1 \leq m \leq n-1\},
		\end{eqnarray*}
		for $n \geq 0$.
		We denote $\widehat{E}^{\pm}(\mathfrak m_{\infty})
		= \bigcup _n \widehat{E}^{\pm}(\mathfrak m_n)$.
		We also define
		\begin{eqnarray*}
		\mathscr C(\mathfrak m_n)
			= \{ P \in \widehat{E}(\mathfrak m_n)
				|\Tr_{n/m+1} P \in \widehat{E}(\mathfrak m_m)\ \text{\rm for all }
				m \equiv n\ (\text{\rm mod } 2), -1 \leq m \leq n-1\}
		\end{eqnarray*}
		for $n \geq 0$ and $\mathscr C(\mathfrak m_{-1})=\widehat{E}(\mathfrak m_{-1})$.
	\end{defn}
	
	By the following lemma,
		it is enough to study $\widehat{E}^{\pm}(\mathfrak m_n)$
		instead of $E^{\pm}(k_n)$ for our purpose.
		
%%%%%%%%%%%%%%%%%%%%%%%%%%%%%%%%%%%%%%%%%%%%%%%%%%%%%%%%%%%%
%%%%		[lemma]	comparison between Epm and E^pm		%%%%
%%%%%%%%%%%%%%%%%%%%%%%%%%%%%%%%%%%%%%%%%%%%%%%%%%%%%%%%%%%%
	\begin{lemma}\label{comparison between Epm and E^pm}
		We can show that the natural maps
			$\widehat{E}^{\pm} (\mathfrak m_n) \rightarrow E^{\pm}(k_n)$
			induce isomorphisms
			$\widehat{E}^{\pm} (\mathfrak m_n) \otimes \mathbb Q_p /\mathbb Z_p
			\overset{\simeq}{\rightarrow} E^{\pm} (k_n) \otimes \mathbb Q_p/\mathbb Z_p$
			for all $n$, and thus we have
		\begin{eqnarray*}
			\widehat{E}^{\pm} (\mathfrak m_{\infty}) \otimes \mathbb Q_p /\mathbb Z_p
			\overset{\simeq}{\longrightarrow}
			E^{\pm} (k_{\infty}) \otimes \mathbb Q_p/\mathbb Z_p.
		\end{eqnarray*}
	\end{lemma}
	
	\begin{proof}
		Indeed, we consider the following commutative diagrams
		\[\xymatrix{
			0 \ar[r]
				& \widehat{E}(\mathfrak m_n) \ar[r]
				& E(k_n) \ar[r]
				& \widetilde{E}(\mathbb F_k) \ar[r]
				& 0\\
			0 \ar[r]
				& \widehat{E}^{\pm}(\mathfrak m_n) \ar[r] \ar@{^{(}->}[u]
				& E^{\pm}(k_n) \ar[r] \ar@{^{(}->}[u]
				& A_n^{\pm} \ar[r] \ar[u]^{\iota _n ^{\pm}}
				& 0
		}\]
		where $\widetilde{E}$ is the reduction of $E$ modulo $p$,
			$\mathbb F_k$ is the residue field of $k$ and
			$A_n^{\pm}$ is the cokernel of
			$\widehat{E}^{\pm}(\mathfrak m_n) \rightarrow E^{\pm}(k_n)$.
		Since $\widehat{E}^{\pm}(\mathfrak m_n)
			= \widehat{E}(\mathfrak m_n) \cap E^{\pm}(k_n)$,
			we see that the right vertical arrows $\iota _n^{\pm}$ are injective.
		Thus $A_n^{\pm}$ are finite as $\widetilde{E}(\mathbb F_k)$ is finite.
		We also note that $A_n^{\pm}[p^{\infty}]=0$,
			since $E/\mathbb Q_p$ has supersingular reduction.
		From the above, our claim will follow immediately.
	\end{proof}
	
	By comparing two definitions, we get the following relations
		between the plus subgroups (the minus subgroups) and the norm subgroups.
%%%%%%%%%%%%%%%%%%%%%%%%%%%%%%%%%%%%%%%%%%%%%%%%%%%%%%%%%%%%%%%%
%%%%		[lemma]	comparison between norm subgroups		%%%%
%%%%%%%%%%%%%%%%%%%%%%%%%%%%%%%%%%%%%%%%%%%%%%%%%%%%%%%%%%%%%%%%
	\begin{lemma}\label{comparison between norm subgroups}
		We have
		\begin{eqnarray*}
			\widehat{E}^+(\mathfrak m_n) =
			\left\{
			\begin{array}{ll}
				\displaystyle
				\mathscr C(\mathfrak m_n) & \text{if } n \text{ is even},\\[5mm]
				\displaystyle
				\mathscr C(\mathfrak m_{n-1}) & \text{if } n \text{ is odd},
			\end{array}
			\right.
		\end{eqnarray*}
		\begin{eqnarray*}
			\widehat{E}^-(\mathfrak m_n) =
			\left\{
			\begin{array}{ll}
				\displaystyle
				\mathscr C(\mathfrak m_n) & \text{if } n \text{ is odd},\\[5mm]
				\displaystyle
				\mathscr C(\mathfrak m_{n-1}) & \text{if } n \text{ is even}.
			\end{array}
			\right.
		\end{eqnarray*}
	\end{lemma}
	
	We now describe $\mathscr C(\mathfrak m_n)$ in terms of the system of local points $(d_n)_n$,
		and thus we get a description of plus and minus subgroups
		$\widehat{E}^{\pm}(\mathfrak m_n)$.
	
%%%%%%%%%%%%%%%%%%%%%%%%%%%%%%%%%%%%%%%%%%%%%%%%%%%%%%%%%%%%
%%%%		[prop]	generators of norm subgroups		%%%%
%%%%%%%%%%%%%%%%%%%%%%%%%%%%%%%%%%%%%%%%%%%%%%%%%%%%%%%%%%%%
	\begin{prop}\label{generators of norm subgroups}
		(1) For each $n \geq -1$, the $n$-th norm subgroup is generated by
		$d_n$ and $d_{-1}$ as a $\mathbb Z_p[G_n]$-module;
%%%%%%%%%%%%%%%%%%%%%%%%%%%%%%%%%%%%%%%%%%%%
%%%%		generators of C(m_n)		%%%%
%%%%%%%%%%%%%%%%%%%%%%%%%%%%%%%%%%%%%%%%%%%%
		\begin{eqnarray*}\label{generators of C(m_n)}
			\mathscr C(\mathfrak m_n) = \langle d_n, d_{-1}\rangle _{\mathbb Z_p[G_n]}.
		\end{eqnarray*}
		
		(2) For each $n \geq 0$, we have an exact sequence
%%%%%%%%%%%%%%%%%%%%%%%%%%%%%%%%%%%%%%%%%%%%%%%%%%%%%%%%
%%%%		exact sequence in norm subgroups		%%%%
%%%%%%%%%%%%%%%%%%%%%%%%%%%%%%%%%%%%%%%%%%%%%%%%%%%%%%%%
		\begin{eqnarray}\label{exact sequence in norm subgroups}
			0 \longrightarrow \widehat{E} (\mathfrak m_{-1})
			\longrightarrow \mathscr C(\mathfrak m_n) \oplus \mathscr C(\mathfrak m_{n-1})
			\longrightarrow \widehat{E} (\mathfrak m_n)
			\longrightarrow 0,
		\end{eqnarray}
		where the first map is diagonal embedding by inclusions,
		and the second map is $(a,b) \mapsto a-b$.
	\end{prop}
	
	\begin{proof}
		We will prove this by the same method as the proof of Proposition 8.12 in \cite{Kob03}.
		The main difference is the element $d_{-1}$ in the first statement.
%		We first note that this element $d_{-1}$ is indispensable to prove
%			both statements (1) and (2).
%		Furthermore, as we will see in Remark \ref{remark on d_{-1}},
%			$\mathscr C (\mathfrak m_n)$ is not a cyclic $\mathbb Z_p[G_n]$-module
%			generated by $d_n$ in general because of $d_{-1}$.
		
		We first show that $\mathscr C (\mathfrak m_n) \cap \mathscr C (\mathfrak m_{n-1})
			= \widehat{E} (\mathfrak m_{-1})$ for $n \geq 0$.
		It is clear that $\mathscr C(\mathfrak m_n)
			\cap \mathscr C (\mathfrak m_{n-1}) \supseteq \widehat{E}(\mathfrak m_{-1})$
			by definition.
		Let $P \in \mathscr C (\mathfrak m_n) \cap \mathscr C (\mathfrak m_{n-1})$.
		We show that if $P \in \widehat{E}(\mathfrak m_m)$ for some $m \geq 0$
			then $P \in \widehat{E}(\mathfrak m_{m-1})$.
		In the case $m \equiv n$ (mod $2$),
			we have $p^{n-m}P = \Tr _{n/m}P \in \widehat{E}(\mathfrak m_{m-1})$
			by the definition of $\mathscr C(\mathfrak m_{n-1})$.
		Therefore for $\sigma \in \Gal (k_m/k_{m-1})$,
			we have $p^{n-m}(P^{\sigma}-P)=0$.
		Hence, by Proposition \ref{E^ is torsion-free},
			$P \in \widehat{E}(\mathfrak m_{m-1})$.
		In the case $m \equiv n-1$ (mod $2$),
			our claim is shown similarly by considering $\mathscr C(\mathfrak m_n)$
			instead of $\mathscr C ({\mathfrak m_{n-1}})$ in the above argument.
		The other inclusion is clearly true.

		For the moment, let $\mathscr C '(\mathfrak m_n)$
			be the $\mathbb Z_p[G_n]$-submodule
			of $\widehat{E}(\mathfrak m_n)$ generated by
			$d_n$ and $d_{-1}$.
		By the trace relations on $d_n$, clearly we have
			$\mathscr C(\mathfrak m_n) \supseteq \mathscr C'(\mathfrak m_n)$.
		We now prove
		\begin{eqnarray*}
			\mathscr C(\mathfrak m_n) = \mathscr C'(\mathfrak m_n),\
			\mathscr C(\mathfrak m_n) + \mathscr C(\mathfrak m_{n-1})
				= \widehat{E}(\mathfrak m_n)
		\end{eqnarray*}
		for $n \geq 0$, simultaneously by induction.
		
		In the case $n=0$, we have
		\begin{eqnarray*}
			&&\mathscr C(\mathfrak m_0) = \widehat{E}(\mathfrak m_0)
			= \langle d_0,d_{-1}\rangle _{\mathbb Z_p[G_0]}
			= \mathscr C'(\mathfrak m_0), \\
			&&\mathscr C(\mathfrak m_0) + \mathscr C(\mathfrak m_{-1})
			= \widehat{E}(\mathfrak m_0) + \widehat{E}(\mathfrak m_{-1})
			= \widehat{E}(\mathfrak m_0)
		\end{eqnarray*}
		by Corollary \ref{generators of E^(m_n)}.
		
		In the case $n \geq 1$, by the induction hypothesis
			we have
		\begin{eqnarray}\label{induction step 1 in generator of norm subgroups}
			\widehat{E}(\mathfrak m_{n-1}) = \mathscr C(\mathfrak m_{n-1})
			+ \mathscr C(\mathfrak m_{n-2}),\
			\mathscr C(\mathfrak m_n) = \mathscr C'(\mathfrak m_n)
		\end{eqnarray}
			and by the trace relation we have
			$\mathscr C'(\mathfrak m_n) \subseteq \mathscr C'(\mathfrak m_n)$.
		Therefore, by Proposition \ref{log isoms 2} and (\ref{induction step 1 in generator of norm subgroups}), we have
		\begin{eqnarray*}
			\widehat{E}(\mathfrak m_n)
			= \langle d_n \rangle_{\mathbb Z_p[G_n]} + \widehat{E}(\mathfrak m_{n-1})
			\subseteq \mathscr C' (\mathfrak m_n) + \mathscr C(\mathfrak m_{n-1}).
		\end{eqnarray*}
		In particular, we have
			$\mathscr C(\mathfrak m_n) \subseteq \mathscr C'(\mathfrak m_n)
			+ \mathscr C (\mathfrak m_{n-1})$.
		This implies that $\mathscr C(\mathfrak m_n) =\mathscr C'(\mathfrak m_n)$.
		Indeed, if $P \in \mathscr C(\mathfrak m_n)$,
			then there exist $Q \in \mathscr C'(\mathfrak m_n)$
			and $R \in \mathscr C(\mathfrak m_{n-1})$ such that $P=Q+R$.
		Then we see that
			$R = P-Q \in \mathscr C(\mathfrak m_n) \cap \mathscr C(\mathfrak m_{n-1})
			= \widehat{E}(\mathfrak m_{-1})$.
		Note that $\widehat{E}(\mathfrak m_{-1}) \subseteq \mathscr C'(\mathfrak m_n)$
			since $d_{-1} \in \mathscr C'(\mathfrak m_n)$.
		So we get $P = Q+R \in \mathscr C'(\mathfrak m_n)$
			and thus $\mathscr C(\mathfrak m_n)= \mathscr C' (\mathfrak m_n)$.
		It is now clear that
			$\mathscr C(\mathfrak m_n) + \mathscr C(\mathfrak m_{n-1})
			= \mathscr C' (\mathfrak m_n)+\mathscr C'(\mathfrak m_{n-1})
			= \widehat{E}(\mathfrak m_n)$.
	\end{proof}
	
%%%%%%%%%%%%%%%%%%%%%%%%%%%%%%%%%%%%%%%%%%%%%%%%
%%%%		[rem]	remark on d_{-1}		%%%%
%%%%%%%%%%%%%%%%%%%%%%%%%%%%%%%%%%%%%%%%%%%%%%%%
	\begin{rem}\label{remark on d_{-1}}
		\rm{We check here that the norm subgroup $\mathscr C(\mathfrak m_n)$
			is not a cyclic $\mathbb Z_p[G_n]$-module generated by $d_n$
			if and only if $d =[k:\mathbb Q_p] \equiv 0$ (mod $4$) and $n$ is even.
		
		(1) When $n$ is odd or $d \nequiv 0$ (mod $4$),
		we see that $d_{-1}$ is automatically contained
		in $\langle d_n \rangle _{\mathbb Z_p[G_n]}$.
		Thus in these cases we see that the norm subgroup $\mathscr C(\mathfrak m_n)$
		is a cyclic $\mathbb Z_p[G_n]$-module generated by $d_n$ for each $n$;
		\begin{eqnarray*}
			\mathscr C (\mathfrak m_n)
			= \langle d_n \rangle _{\mathbb Z_p[G_n]}.
		\end{eqnarray*}
		Indeed, when $n$ is odd, we have
		\begin{eqnarray*}
			d_{-1} = (-1)^{\frac{n+1}{2}} \Tr _{1/0} \cdots \Tr_{n-2/n-3} \Tr _{n/n-1}d_n
				\in \langle d_n \rangle_{\mathbb Z_p[G_n]}.
		\end{eqnarray*}
		When $d \nequiv 0$ (mod $4$) and $n$ is even, we have
		\begin{eqnarray*}
			d_{-1} = (-1)^{\frac{n+2}{2}} (\varphi + \varphi ^{-1})^{-1}\Tr _{0/-1} \cdots
				\Tr_{n-2/n-3} \Tr _{n/n-1}d_n
				\in \langle d_n \rangle _{\mathbb Z_p[G_n]}, 
		\end{eqnarray*}
		since $\varphi + \varphi ^{-1} \in \mathbb Z_p[G_{-1}]^{\times}$
			by lemma \ref{on varphi + varphi^{-1} and d}.

		(2) When $d \equiv 0$ (mod $4$),
		$d_{-1}$ cannot be contained in $\langle d_n \rangle _{\mathbb Z_p[G_n]}$
			for any even $n$.
		Thus in this case the norm subgroup $\mathscr C (\mathfrak m_n)$ is
			not a cyclic $\mathbb Z_p[G_n]$-module
			generated by $d_n$ for each even $n$.
		Indeed, if $\widehat{E}(\mathfrak m_0) = \langle d_0\rangle_{\mathbb Z_p[G_0]}$,
			then $\widehat{E}(\mathfrak m_{-1})
			= \langle \Tr_{0/-1}(d_0) \rangle _{\mathbb Z_p[G_{-1}]}$
			since $\Tr_{0/-1}:\widehat{E}(\mathfrak m_0)
			\rightarrow \widehat{E}(\mathfrak m_{-1})$ is surjective.
		Since $\widehat{E}(\mathfrak m_{-1}) \cong \mathbb Z_p[G_{-1}]$,
			this means that $\mathbb Z_p[G_{-1}]
			= (\varphi + \varphi ^{-1}) \mathbb Z_p[G_{-1}]$,
			which is impossible by lemma \ref{on varphi + varphi^{-1} and d}.}
	\end{rem}

%%%%%%%%%%%%%%%%%%%%%%%%%%%%%%%%%%%%%%%%%%%%%%%%%%%%%%%%%%%%%%%%%%%%
%%%%		[defn]	plus and minus local points d_n^{pm}		%%%%
%%%%%%%%%%%%%%%%%%%%%%%%%%%%%%%%%%%%%%%%%%%%%%%%%%%%%%%%%%%%%%%%%%%%
	\begin{defn}
		Define $d_n^{\pm}$ by
		\begin{eqnarray*}
			d_n^+ =\left\{
			\begin{array}{ll}
				(-1)^{\frac{n+2}{2}}d_n & \text{if }n \text{ is even},\\[3mm]
				(-1)^{\frac{n+1}{2}}d_{n-1} & \text{if }n \text{ is odd},
			\end{array}
			\right.
			d_n^- =\left\{
			\begin{array}{ll}
				(-1)^{\frac{n+1}{2}}d_n & \text{if }n \text{ is odd},\\[3mm]
				(-1)^{\frac{n}{2}}d_{n-1} & \text{if }n \text{ is even}.
			\end{array}
			\right.
		\end{eqnarray*}
	\end{defn}

%%%%%%%%%%%%%%%%%%%%%%%%%%%%%%%%%%%%%%%%%%%%%%%%%%%%%%%%%%%%%%%%%%%%%%%%
%%%%		[rem]	plus minus subgroups in terms of d_n^{pm}		%%%%
%%%%%%%%%%%%%%%%%%%%%%%%%%%%%%%%%%%%%%%%%%%%%%%%%%%%%%%%%%%%%%%%%%%%%%%%
	\begin{rem}\label{plus minus subgroups in termes of d_n^{pm}}
		\rm{By the relation between $\mathscr C(\mathfrak m_n)$ and
			$\widehat{E}^{\pm}(\mathfrak m_n)$
			(cf. Lemma \ref{comparison between norm subgroups}),
			we can translate Proposition \ref{generators of norm subgroups}
			in terms of the plus and the minus systems of points $(d_n^+)_n$ and $(d_n^-)_n$ such that
			$\widehat{E}^{+}(\mathfrak m_n)
			= \langle d_n^+ , d_0^-\rangle_{\mathbb Z_p[G_n]}$
			and $\widehat{E}^{-}(\mathfrak m_n)
			= \langle d_n^- , d_0^-\rangle_{\mathbb Z_p[G_n]}$.
		As in Remark \ref{remark on d_{-1}},
			we see that the the plus subgroups $\widehat{E}^+(\mathfrak m_n)$
			are cyclic $\mathbb Z_p[G_n]$-modules generated by $d_n^+$ for all $n$
			if and only if $d \nequiv 0$ (mod $4$),
			on the other hand
			the minus subgroups $\widehat{E}^-(\mathfrak m_n)$
			are always cyclic $\mathbb Z_p[G_n]$-modules generated by $d_n^-$ for all $n$.}
%		Together with the remark \ref{remark on d_{-1}}, we see that
%		the plus (resp. minus) subgroups are cyclic $\mathbb Z_p[G_n]$-modules
%		in some special cases. Indeed we have
%		\begin{eqnarray*}
%			\widehat{E}^+(\mathfrak m_n) = \langle d_n^+\rangle _{\mathbb Z_p[G_n]}
%		\end{eqnarray*}
%		if $d \nequiv 0$ (mod $4$), and
%		\begin{eqnarray*}
%			\widehat{E}^-(\mathfrak m_n) = \langle d_n^- \rangle _{\mathbb Z_p[G_n]}.
%		\end{eqnarray*}
	\end{rem}

%%%%%%%%%%%%%%%%%%%%%%%%%%%%%%%%%%%%%%%%%%%%%%%%%%%%%%%%%%%%%%%%
%%%%		character decomposition	of a Delta-module		%%%%
%%%%%%%%%%%%%%%%%%%%%%%%%%%%%%%%%%%%%%%%%%%%%%%%%%%%%%%%%%%%%%%%
	Let $\chi:\Delta \rightarrow \mathbb Z_p^{\times}$
		be a character of $\Delta = \Gal (k(\mu_p)/k)$.
	If $M$ is a $\mathbb Z_p[\Delta]$-module,
		then $M$ is decomposed into
		\begin{eqnarray*}
			M = \bigoplus _{\chi}\varepsilon _{\chi} M,
		\end{eqnarray*}
		where $\varepsilon _{\chi}
		= \frac{1}{p-1} \sum _{\sigma \in \Delta} \chi (\sigma) \sigma ^{-1}
		\in \mathbb Z_p[\Delta]$.
	We denote by $M^{\chi}$ the $\chi$-component $\varepsilon_{\chi}M$.
	
	Since we have $G_n \cong G_{-1} \times \Delta \times \Gal(k_n/k_0)$,
		we can regard a $\mathbb Z_p[G_n]$-module
		as a $\mathbb Z_p[\Delta]$-module.
	
%%%%%%%%%%%%%%%%%%%%%%%%%%%%%%%%%%%%%%%%%%%%%%%%%%%%%%%%
%%%%		[cor]	Z_p-ranks of norm subgroups		%%%%
%%%%%%%%%%%%%%%%%%%%%%%%%%%%%%%%%%%%%%%%%%%%%%%%%%%%%%%%
	\begin{cor}\label{Z_p-ranks of norm subgroups}
		Let $\chi: \Delta \rightarrow \mathbb Z_p^{\times}$ be a character
			and $q_n = \sum _{i=0}^n(-1)^ip^{n-i}$.
		Then we have
		\begin{eqnarray*}
			\rank _{\mathbb Z_p} \mathscr C(\mathfrak m_n)^{\chi}
			=\left\{
			\begin{array}{ll}
				\displaystyle
				d(q_n+1) 
				& \text{ if } n:\text{odd and } \chi = \mathbf 1, \\[2mm]
				\displaystyle
				d q_n
				& \text{ otherwise},
			\end{array}
			\right.
		\end{eqnarray*}
		for each $n \geq 0$ and 
		\begin{eqnarray*}
			\rank _{\mathbb Z_p} \mathscr C(\mathfrak m_{-1})^{\chi}
			=\left\{
			\begin{array}{ll}
				d & \text{ if } \chi = \mathbf 1,\\[1mm]
				0 & \text{ if } \chi \neq \mathbf 1.
			\end{array}
			\right.
		\end{eqnarray*}
	\end{cor}
	
	\begin{proof}
		Since $\mathscr C(\mathfrak m_{-1}) = \widehat{E} (\mathfrak m_{-1})
			\cong \mathbb Z_p[G_{-1}]$,
			we obtain the latter statement.
		Moreover,
			from the exact sequence (\ref{exact sequence in norm subgroups})
			we get a recurrence sequence
			\begin{eqnarray*}
				&& \rank _{\mathbb Z_p} \mathscr C (\mathfrak m_n)^{\chi}
					+ \rank _{\mathbb Z_p} \mathscr C (\mathfrak m_{n-1})^{\chi}\\
				&& =\rank _{\mathbb Z_p} \widehat{E} (\mathfrak m_{-1})^{\chi}
					+ \rank _{\mathbb Z_p} \widehat{E} (\mathfrak m_n)^{\chi},
			\end{eqnarray*}
		since the $\mathbb Z_p$-rank of $\widehat{E}(\mathfrak m_n)^{\chi}$ is $dp^n$ for each $n \geq 0$.
		The former statement follows from this.
	\end{proof}

%%%%%%%%%%%%%%%%%%%%%%%%%%%%%%%%%%%%
%%%%		some notations		%%%%
%%%%%%%%%%%%%%%%%%%%%%%%%%%%%%%%%%%%
	We introduce here some notation
		that will be used throughout the rest of the paper.
	Let $\omega _n(X) := (1+X)^{p^n}-1$ and
		$\Phi _n (X) := \sum _{i=0}^{p-1} X^{ip^{n-1}}$
		be the $p^n$-th cyclotomic polynomial.
	We define $\widetilde{\omega} _0^{\pm}(X):=1$ and
	\begin{eqnarray*}
		\displaystyle
		&&\widetilde{\omega} _n^+(X)
			= \prod _{1 \leq m \leq n, m:\text{even}} \Phi _m(1+X), \quad
		\omega _n^+(X) = X \widetilde{\omega} _n^+(X),\\
		&&\widetilde{\omega} _n^-(X)
			= \prod _{1 \leq m \leq n, m:\text{odd}} \Phi _m(1+X), \quad
		\omega _n^-(X) = X \widetilde{\omega} _n^-(X).
	\end{eqnarray*}
	Note that $\omega _n(X) = \widetilde{\omega} _n^{\mp} (X) \omega _n^{\pm}(X)$ for all $n \geq 0$.
	We write $\omega _n(X)$, $\widetilde{\omega} _n^{\pm}(X)$ and $\omega _n^{\pm}(X)$
		simply by $\omega_n$, $\widetilde{\omega}_n^{\pm}$ and $\omega _n^{\pm}$ respectively.
	
	We identify $\mathbb Z_p[G_n]$ with $\mathbb Z_p[G_0][X]
		/\langle \omega _n \rangle _{\mathbb Z_p[G_0][X]}$
		by sending $\gamma _n$ to $1+X$,
		where $\gamma _n$ is the image of $\gamma$ in $\mathbb Z_p[G_n]$.
	
	Set $q_n = \sum _{i=0}^n (-1)^{n-i}p^i$ as in Corollary \ref{Z_p-ranks of norm subgroups}
		and $q_{-1}:=1$.
	Put
	$$
		q_n^+ := \left\{
		\begin{array}{ll}
			q_n & \text{ if } n \text{ is even},\\[3mm]
			q_{n-1} & \text{ if } n \text{ is odd},
		\end{array}
		\right.
		q_n^- := \left\{
		\begin{array}{ll}
			q_n & \text{ if } n \text{ is odd},\\[3mm]
			q_{n-1} & \text{ if } n \text{ is even}.
		\end{array}
		\right.
	$$
	Note that $q_n^+ + q_n^- = p^n$ for all $n \geq 0$.
	
	For later use,
		we rephrase Corollary \ref{Z_p-ranks of norm subgroups}
		in terms of $\widehat{E}^{\pm}(\mathfrak m_n)^{\chi}$ and $q_n^{\pm}$
		as in the following corollary.
%%%%%%%%%%%%%%%%%%%%%%%%%%%%%%%%%%%%%%%%%%%%%%%%%%%%%%%%%%%%%%%%%%%%%
%%%%		[cor]	Z_p-ranks of plus and minus subgroups		%%%%%
%%%%%%%%%%%%%%%%%%%%%%%%%%%%%%%%%%%%%%%%%%%%%%%%%%%%%%%%%%%%%%%%%%%%%
	\begin{cor}\label{Z_p-ranks of plus and minus subgroups}
		Let $\chi : \Delta \rightarrow \mathbb Z_p^{\times}$ be a character.
		Then we have
		\begin{eqnarray*}
			\rank _{\mathbb Z_p} (\widehat{E}^+ (\mathfrak m_n) ^{\chi})
			&=& d q_n^+, \\[1mm]
			\rank _{\mathbb Z_p} (\widehat{E}^- (\mathfrak m_n)^{\chi})
			&=& \left\{
			\begin{array}{lll}
				d (q_n^- + 1) & \text{if } \chi = \mathbf 1, \\[3mm]
				d q_n^- & \text{if } \chi \neq \mathbf 1.
			\end{array}
			\right.
		\end{eqnarray*}
	\end{cor}

	For a character $\chi$ of $\Delta$,
		we define $\mathbb Z_p[\chi]$
		to be the $\mathbb Z_p [\Delta]$-module
		which is $\mathbb Z_p$ as a $\mathbb Z_p$-module,
		and on which $\Delta$ acts via $\chi$,
		namely $\sigma \cdot x = \chi (\sigma) x$ for $\sigma \in \Delta$ and $x \in \mathbb Z_p[\chi]$.
	
%%%%%%%%%%%%%%%%%%%%%%%%%%%%%%%%%%%%%%%%%%%%%%%%%%%%%%%%%%%%%%%%%%%%%%%%%%%%%%%%
%%%%		[prop]	galois module structure of plus and minus subgroups		%%%%
%%%%%%%%%%%%%%%%%%%%%%%%%%%%%%%%%%%%%%%%%%%%%%%%%%%%%%%%%%%%%%%%%%%%%%%%%%%%%%%%
	\begin{prop}\label{galois module structure of plus minus subgroups}
		Let $\chi:\Delta \rightarrow \mathbb Z_p^{\times}$ be a character.
		We have
		\begin{eqnarray*}
			\widehat{E}^+(\mathfrak m_n)^{\chi} & \cong &
			\left\{
			\begin{array}{ll}
				\displaystyle
				\frac{ \mathbb Z_p[G_{-1}][X] \oplus \mathbb Z_p[G_{-1}]}
					{\langle (\widetilde{\omega}_n^+,
					-(\varphi + \varphi ^{-1}))\rangle _{\mathbb Z_p[G_{-1}][X]}}
					& \text{ if } \chi = \mathbf 1, \\[5mm]
				\displaystyle
				\mathbb Z_p[\chi] \otimes_{\mathbb Z_p}
				\frac{\mathbb Z_p[G_{-1}][X]}
				{\langle \omega _n^+\rangle _{\mathbb Z_p[G_{-1}][X]}}
					& \text{ if } \chi \neq \mathbf 1,
			\end{array}
			\right. \\[3mm]
			\widehat{E}^-(\mathfrak m_n)^{\chi} & \cong &
			\left\{
			\begin{array}{ll}
				\displaystyle
				\mathbb Z_p[G_{-1}][X]/\langle \omega _n^-\rangle
				_{\mathbb Z_p[G_{-1}][X]}
					& \text{ if } \chi = \mathbf 1, \\[3mm]
				\displaystyle
				\mathbb Z_p[\chi] \otimes_{\mathbb Z_p}
				\frac{\mathbb Z_p[G_{-1}][X]}{\langle \widetilde{\omega}^-_n \rangle
				_{\mathbb Z_p[G_{-1}][X]}}
					& \text{ if } \chi \neq \mathbf 1,
			\end{array}
			\right.
		\end{eqnarray*}
		as $\mathbb Z_p[G_n]$-modules.
	\end{prop}
	
	\begin{proof}
		There is a surjective homomorphism
		$$
			\psi : \frac{\mathbb Z_p[G_0][X]}
			{\langle \omega _n \rangle _{\mathbb Z_p[G_0][X]}}
			\oplus \mathbb Z_p[G_{-1}] \longrightarrow
			\langle d_n^+, d_0^- \rangle _{\mathbb Z_p[G_n]}
			= \widehat{E}^+ (\mathfrak m_n)
		$$
		obtained by sending $(1,0)$ to $d_n^+$ and $(0,1)$ to $\frac{d_0^-}{p-1} (= \frac{d_{-1}}{p-1})$.
		Since we have that
		\begin{eqnarray*}
			&&\omega _n^+ d_n^+= \omega _{n-2}^+ d_{n-2}^+ = \cdots = 0,\ \text{ and } \\
			&& \varepsilon_{\mathbf 1} \widetilde{\omega}_n^+ d_n^+
			= \frac{1}{p-1}\Tr _{0/-1}\widetilde{\omega}_n^+ d_n^+
			= (\varphi + \varphi ^{-1}) \frac{d_0^-}{p-1},
		\end{eqnarray*}
		the map $\psi$ induces a surjective homomorphism
		\begin{eqnarray*}
			\overline{\psi}: \frac{\frac{\mathbb Z_p[G_0][X]}
			{\langle \omega _n^+ \rangle _{\mathbb Z_p[G_0][X]}}
			\oplus \mathbb Z_p[G_{-1}]}
			{\langle (\varepsilon_{\mathbf 1} \widetilde{\omega}_n^+,-(\varphi + \varphi ^{-1}))
			\rangle _{\mathbb Z_p[G_0][X]}}
			\longrightarrow \langle d_n^+ , d_0^-\rangle _{\mathbb Z_p[G_n]}
			= \widehat{E}^+ (\mathfrak m_n).
		\end{eqnarray*}
		This map $\overline{\psi}$ is injective
			since the source
%			$\left( \frac{\mathbb Z_p[G_0][X]}
%			{\langle \omega _n^+\rangle _{\mathbb Z_p[G_0][X]}}
%			\oplus \mathbb Z_p[G_{-1}] \right)
%			/\langle (\Tr _{0/-1}\widetilde{\omega}_n^+ ,-(\varphi + \varphi ^{-1}))
%			\rangle _{\mathbb Z_p[G_0][X]}$
			and the target of $\overline{\psi}$
%			$\langle d_n^+, d_0^- \rangle _{\mathbb Z_p[G_n]}$
			are free $\mathbb Z_p$-modules of the same $\mathbb Z_p$-rank $d(p-1)q_n^+$
			(cf. Corollary \ref{Z_p-ranks of plus and minus subgroups}).
		Thus we have
		\begin{eqnarray*}
			\widehat{E}^+(\mathfrak m_n)
			& \cong & \frac{\frac{\mathbb Z_p[G_0][X]}
				{\langle \omega _n^+\rangle _{\mathbb Z_p[G_0][X]}}
				\oplus \mathbb Z_p[G_{-1}]}
				{\left\langle (\varepsilon_{\mathbf 1} \widetilde{\omega}_n^+,
				-(\varphi + \varphi ^{-1}))
				\right\rangle _{\mathbb Z_p[G_0][X]}}\\[3mm]
			& \cong & \frac{\mathbb Z_p[G_0][X] \oplus \mathbb Z_p[G_{-1}]}
				{\langle (\omega _n^+,0),
				(\varepsilon_{\mathbf 1} \widetilde{\omega}_n^+, -(\varphi + \varphi ^{-1}))
				\rangle _{\mathbb Z_p[G_0][X]}}\\[3mm]
			& \cong & \bigoplus _{\chi} \frac{ \varepsilon_\chi \mathbb Z_p[G_0][X]
				\oplus \varepsilon_\chi \mathbb Z_p[G_{-1}]}
				{\langle (\varepsilon_\chi \omega _n^+,0),
				(\varepsilon_\chi \varepsilon_{\mathbf 1} \widetilde{\omega}_n^+,
				-\varepsilon_\chi (\varphi + \varphi ^{-1}))
				\rangle _{\mathbb Z_p[G_0][X]}}
		\end{eqnarray*}
		as $\mathbb Z_p[G_n]$-modules,
			where the last isomorphism is obtained by the character decomposition.
		Since we have
		$$
			\varepsilon_\chi \mathbb Z_p[G_0][X] \oplus \varepsilon_\chi \mathbb Z_p[G_{-1}]
			\cong
			\left\{
			\begin{array}{ll}
				\displaystyle
				\mathbb Z_p[G_{-1}][X] \oplus \mathbb Z_p[G_{-1}]
					& \text{ if } \chi = \mathbf 1,\\[3mm]
					\displaystyle
				\mathbb Z_p[G_{-1}][X]
					& \text{ if } \chi \neq \mathbf 1,
			\end{array}
			\right.
		$$
		and
		\begin{eqnarray*}
			&&\langle (\varepsilon_\chi \omega _n^+,0),
			(\varepsilon_\chi \varepsilon_{\mathbf 1} \widetilde{\omega}_n^+,
			-\varepsilon_\chi (\varphi + \varphi ^{-1}))
			\rangle _{\mathbb Z_p[G_0][X]}\\[2mm]
			&&\cong
			\left\{
			\begin{array}{ll}
				\displaystyle
				\langle (\omega _n^+,0),
				(\widetilde{\omega}_n^+, -(\varphi + \varphi ^{-1}))
				\rangle_{\mathbb Z_p[G_{-1}][X]}
					& \text{ if } \chi = \mathbf 1, \\[3mm]
				\displaystyle
				\langle \omega _n^+
				\rangle_{\mathbb Z_p[G_{-1}][X]}
					& \text{ if } \chi \neq \mathbf 1
			\end{array}
			\right.
		\end{eqnarray*}
		as $\mathbb Z_p[G_{-1}][X]$-modules,
		we have
		\begin{eqnarray*}
			\widehat{E}^+(\mathfrak m_n) ^{\chi}
			& \cong & \frac{ \varepsilon_\chi \mathbb Z_p[G_0][X]
				\oplus \varepsilon_\chi \mathbb Z_p[G_{-1}]}
				{\langle (\varepsilon_\chi \omega _n^+,0),
				(\varepsilon_\chi \varepsilon_{\mathbf 1} \widetilde{\omega}_n^+,
				-\varepsilon_\chi (\varphi + \varphi ^{-1}))
				\rangle _{\mathbb Z_p[G_0][X]}}\\[3mm]
			& \cong & \left\{
			\begin{array}{ll}
				\displaystyle
				\frac{\mathbb Z_p[G_{-1}][X] \oplus \mathbb Z_p[G_{-1}]}
				{\langle (\omega _n^+,0),
				(\widetilde{\omega}_n^+, -(\varphi + \varphi ^{-1}))
				\rangle _{\mathbb Z_p[G_{-1}][X]}} &  \text{ if } \chi = \mathbf 1,\\[5mm]
				\displaystyle
				\mathbb Z_p[\chi] \otimes_{\mathbb Z_p}
				\frac{\mathbb Z_p[G_{-1}][X]}{\langle
				\omega _n^+ \rangle _{\mathbb Z_p[G_{-1}][X]}} &  \text{ if } \chi \neq \mathbf 1
			\end{array}\right.
		\end{eqnarray*}
		as $\mathbb Z_p[G_n]$-modules.
		Since $(\omega _n^+,0)
			= X(\widetilde{\omega}_n^+,-(\varphi + \varphi ^{-1}))$,
			we get the conclusion for $\widehat{E}^+(\mathfrak m_n)^{\chi}$.\\
				
%%%%%%%%%%%%%%%%%%%%%%%%%%%%%%%%%%%%
%%%%		minus subgroup		%%%%
%%%%%%%%%%%%%%%%%%%%%%%%%%%%%%%%%%%%
		Similarly to the above, we have
		\begin{eqnarray*}
			\widehat{E}^-(\mathfrak m_n)
			&=& \langle d_n^- \rangle _{\mathbb Z_p[G_n]}\\[3mm]
			& \cong & \mathbb Z_p[G_0][X] / \langle
				\omega _n^-, (\sigma -1)\widetilde{\omega _n}^-
				| \sigma \in \Delta \rangle _{\mathbb Z_p[G_0][X]}\\[3mm]
			& \cong & \frac{\mathbb Z_p[G_{-1}][X]}{\langle
				\omega _n^-
				\rangle _{\mathbb Z_p[G_{-1}][X]}}
				\oplus \bigoplus _{\chi \neq \mathbf 1}
				\left( \mathbb Z_p[\chi] \otimes _{\mathbb Z_p}
				\frac{\mathbb Z_p[G_{-1}][X]}
				{\langle \widetilde{\omega}_n^-
				\rangle _{\mathbb Z_p[G_{-1}][X]}}
				\right)
		\end{eqnarray*}
		as $\mathbb Z_p[G_n]$-modules.
		So we get the conclusion for $\widehat{E}^-(\mathfrak m_n)^{\chi}$.
	\end{proof}
	
%%%%%%%%%%%%%%%%%%%%%%%%%%%%%%%%%%%%%%%%%%%%%%%%%%%%%%%%%%%%%%%%%%%%
%%%%		[rem]	remark on the case d nequiv 0 (mod 4)		%%%%
%%%%%%%%%%%%%%%%%%%%%%%%%%%%%%%%%%%%%%%%%%%%%%%%%%%%%%%%%%%%%%%%%%%%
	\begin{rem}\label{remark on the case d nequiv 0 (mod 4)}
		\rm{When $d \nequiv 0$ (mod $4$) and $\chi = \mathbf 1$,
			the description of the Galois module $\widehat{E}^+(\mathfrak m_n)^{\chi}$
			in Proposition \ref{galois module structure of plus minus subgroups}
			can be made more simpler.
		Explicitly, we claim that the homomorphism
		\begin{eqnarray*}
			\mathbb Z_p[G_{-1}][X] / \langle \omega _n^+\rangle
			\overset{\simeq}{\longrightarrow}
				\frac{\mathbb Z_p[G_{-1}][X] \oplus \mathbb Z_p[G_{-1}]}
				{\langle (\widetilde{\omega}_n^+,-(\varphi + \varphi ^{-1}))\rangle}
		\end{eqnarray*}
		given by $x \mapsto (x,0)$ is an isomorphism.
		Indeed, since $\varphi + \varphi ^{-1} \in \mathbb Z_p[G_{-1}]^{\times}$
			in this case (see Lemma \ref{on varphi + varphi^{-1} and d}),
			$(x,y) \in \mathbb Z_p[G_{-1}][X]\oplus \mathbb Z_p[G_{-1}]$ is equivalent to
			$(x+y(\varphi + \varphi ^{-1})^{-1}\widetilde{\omega}_n^+,0)$
			and thus the map is surjective.
		On the other hand, if $(x,0) \in \langle
			(\widetilde{\omega}_n^+,-(\varphi + \varphi ^{-1}))\rangle$
			for $x \in \mathbb Z_p[G_{-1}][X]$,
			then there exists $a(X) \in \mathbb Z_p[G_{-1}][X]$ such that
			$x = a(X) \widetilde{\omega}_n^+$ and $0 = -a(0) (\varphi + \varphi^{-1})$.
		Again by Lemma \ref{on varphi + varphi^{-1} and d}, we see that $a(0)=0$. So we get
			$x = \frac{a(X)}{X}\omega _n^+ \in \langle \omega _n^+\rangle$
			and thus the map is injective.}
	\end{rem}
	
	In the rest of this paper,
		we abbreviate $\mathbb Z_p[G_{-1}][X]$-modules
		$\langle S \rangle_{\mathbb Z_p[G_{-1}][X]}$ generated by some set $S$
		to $\langle S \rangle$
		as in the above remark.

%%%%%%%%%%%%%%%%%%%%%%%%%%%%%%%%%%%%%%%%%%%%%%%%%%%%%%%%%%%%%%%%
%%%%%%%%%%%%%%%%%%%%%%%%%%%%%%%%%%%%%%%%%%%%%%%%%%%%%%%%%%%%%%%%
%%%%	§3.3 §3.3 §3.3 	%%%%%%%%	§3.3 §3.3 §3.3	%%%%
%%%%%%%%%%%%%%%%%%%%%%%%%%%%%%%%%%%%%%%%%%%%%%%%%%%%%%%%%%%%%%%%
%%%%%%%%%%%%%%%%%%%%%%%%%%%%%%%%%%%%%%%%%%%%%%%%%%%%%%%%%%%%%%%%
\subsection{The plus and the minus local conditions}
	In this subsection, we study the $\Lambda$-module
		$(\widehat{E}^{\pm}(\mathfrak m_{\infty})^{\chi} \otimes \mathbb Q_p/\mathbb Z_p)^{\vee}$
		and prove Proposition \ref{structure of the plus and the minus local conditions}.
	We also study the $\Lambda$-module
		$$\left( \frac{H^1(k_{\infty},E[p^{\infty}])}
		{\widehat{E}^{\pm}(\mathfrak m_{\infty}) \otimes \mathbb Q_p/\mathbb Z_p } \right)^{\vee}$$
		and prove Proposition \ref{Lambda module structure of H^1/E^{pm}}.\\
	
	We first study $(\widehat{E}^{\pm}(\mathfrak m_{\infty})^{\chi} \otimes \mathbb Q_p/\mathbb Z_p)^{\vee}$.
	
	Since $\widehat{E}^{\pm}(\mathfrak m_{\infty})$ are $\mathbb Z_p$-torsion-free,
		we have an exact sequence
	\begin{eqnarray*}
		0  \longrightarrow
		\widehat{E}^{\pm} (\mathfrak m_{\infty})^{\chi} \longrightarrow
		\widehat{E}^{\pm} (\mathfrak m_{\infty})^{\chi}
			\otimes \mathbb Q_p \longrightarrow
		\widehat{E}^{\pm} (\mathfrak m_{\infty})^{\chi}
			\otimes \mathbb Q_p/\mathbb Z_p \longrightarrow
		0.
	\end{eqnarray*}
	From this exact sequence,
		we get the $\Gamma _n$-invariant-coinvariant exact sequence
%%%%%%%%%%%%%%%%%%%%%%%%%%%%%%%%%%%%%%%%%%%%%%%%%%%%%%%%%%%%%%%%%%%%%%%%%%%%%%%%%%%%
%%%%		invariant coinvariant exact seqn's concerning local conditions		%%%%
%%%%%%%%%%%%%%%%%%%%%%%%%%%%%%%%%%%%%%%%%%%%%%%%%%%%%%%%%%%%%%%%%%%%%%%%%%%%%%%%%%%%
	\begin{eqnarray}\label{invariant coinvariant exact seqn's concerning local conditions}
		0
		\longrightarrow 
			\widehat{E}^{\pm}(\mathfrak m_n)^{\chi} \otimes \mathbb Q_p/\mathbb Z_p
		\longrightarrow \left(
			\widehat{E}^{\pm}(\mathfrak m_{\infty})^{\chi}
			\otimes \mathbb Q_p/\mathbb Z_p \right) ^{\Gamma _n}
		\nonumber \\[1mm]
		\longrightarrow \left(
			\widehat{E}^{\pm}(\mathfrak m_{\infty}) ^{\chi} \right) _{\Gamma _n} [p^{\infty}]
		\longrightarrow 0,
	\end{eqnarray}
	for each $n \geq 0$.
	We will compute the rightmost modules $\left( \widehat{E}^{\pm}
		(\mathfrak m_{\infty})^{\chi}\right) _{\Gamma_n}[p^{\infty}]$
		for all $n \geq 0$
		to study the $\Lambda$-module
		$(\widehat{E}^{\pm}(\mathfrak m_{\infty})^{\chi}
		\otimes \mathbb Q_p/\mathbb Z_p)^{\vee}$.
		
	Define $\delta$ by
	\begin{eqnarray*}
		\delta =
		\left\{
		\begin{array}{ll}
			0 & \text{ if } d \nequiv 0\ (\text{mod } 4)
				\text{ or } \chi \neq \mathbf 1,\\[2mm]
			2 & \text{ otherwise}.
		\end{array}
		\right.
	\end{eqnarray*}.

%%%%%%%%%%%%%%%%%%%%%%%%%%%%%%%%%%%%%%%%%%%%%%%%
%%%%		[prop]	pre-key proposition		%%%%
%%%%%%%%%%%%%%%%%%%%%%%%%%%%%%%%%%%%%%%%%%%%%%%%
	\begin{prop}\label{pre-key proposition}
		Let $\chi : \Delta \rightarrow \mathbb Z_p^{\times}$ be a character.
		Then $\left( ( \widehat{E}^{\pm}(\mathfrak m_{\infty}) ^{\chi} ) _{\Gamma _n} [p^{\infty}] \right)^{\vee}$
			are free $\mathbb Z_p$-modules for all $n$, and we have
		\begin{eqnarray*}
			\rank _{\mathbb Z_p}
			\left( ( \widehat{E}^+(\mathfrak m_{\infty}) ^{\chi} ) _{\Gamma _n} [p^{\infty}] \right)^{\vee}
			&=& d q_n^- + \delta ,\\[3mm]
			\rank _{\mathbb Z_p}
			\left( ( \widehat{E}^-(\mathfrak m_{\infty}) ^{\chi} ) _{\Gamma _n} [p^{\infty}] \right)^{\vee}
			&=& \left\{
			\begin{array}{ll}
				d (q_n^+ -1) & \text{ if } \chi = \mathbf 1,\\[3mm]
				d q_n^+ & \text{ if } \chi \neq \mathbf 1.
			\end{array}\right.
		\end{eqnarray*}
		
		More precisely, we have
		\begin{eqnarray*}
			&& \left( \widehat{E}^+(\mathfrak m_{\infty}) ^{\chi} \right) _{\Gamma _n}[p^{\infty}]\\[1mm]
			&& \qquad \cong \left\{
				\begin{array}{ll}
				\left(
				\frac{\mathbb Z_p[G_{-1}][X]}
				{\langle \widetilde{\omega} _n^- \rangle}
				\oplus \Ann _{\mathbb Z_p[G_{-1}]}(\varphi + \varphi ^{-1})
				\right)
				\otimes \mathbb Q_p/\mathbb Z_p & \text{ if } \chi = \mathbf 1,\\[5mm]
				\left(
				\frac{\mathbb Z_p[G_{-1}][X]}{\langle \widetilde{\omega }_n^- \rangle}
				\right) \otimes \mathbb Q_p/\mathbb Z_p & \text{ if } \chi \neq \mathbf 1,
				\end{array}
			\right. \\[3mm]
			&& \left( \widehat{E}^-(\mathfrak m_{\infty}) ^{\chi} \right) _{\Gamma _n} [p^{\infty}]\\[1mm]
			&& \qquad \cong  \left\{
				\begin{array}{ll}
					\displaystyle
					\frac{\mathbb Z_p[G_{-1}][X]}
					{\langle \widetilde{\omega} _n^+ \rangle}
					\otimes \mathbb Q_p/\mathbb Z_p & \text{ if } \chi = \mathbf 1,\\[5mm]
					\displaystyle
					\frac{\mathbb Z_p[G_{-1}][X]}
					{\langle \omega _n^+ \rangle}
					\otimes \mathbb Q_p/\mathbb Z_p & \text{ if } \chi \neq \mathbf 1
				\end{array}
			\right.
		\end{eqnarray*}
		as $\mathbb Z_p$-modules.
	\end{prop}
	
	\begin{proof}
		We prove the claim for
			$\left( \widehat{E}^+(\mathfrak m_{\infty}) ^{\chi} \right) _{\Gamma _n} [p^{\infty}]$
			in the case where $\chi = \mathbf 1$.
		We can prove the rest of the claims similarly.
		
		We have
%%%%%%%%%%%%%%%%%%%%%%%%%%%%%%%%%%%%%%%%%%%%%%%%%%%%%%%%
%%%%		ingredient 1 in pre-key proposition		%%%%
%%%%%%%%%%%%%%%%%%%%%%%%%%%%%%%%%%%%%%%%%%%%%%%%%%%%%%%%
			\begin{eqnarray}\label{ingredient 1 in pre-key proposition}
				\left( \widehat{E}^+(\mathfrak m_{\infty}) ^{\chi} \right) _{\Gamma _n} [p^{\infty}]
				&=& \left (\widehat{E}^+(\mathfrak m_{\infty})^{\chi}
					/\omega _n \widehat{E}^+(\mathfrak m_{\infty})^{\chi} \right) [p^{\infty}]
					\nonumber \\[3mm]
				& \cong & \underset{\substack{\longrightarrow \\ m \geq n}}{\lim}\
					\left( \widehat{E}^+(\mathfrak m_m)^{\chi}
					/\omega _n \widehat{E}^+(\mathfrak m_m)^{\chi} \right) [p^{\infty}],
			\end{eqnarray}
			where transition maps
			$$
				\left( \widehat{E}^+(\mathfrak m_m)^{\chi}
					/\omega _n \widehat{E}^+(\mathfrak m_m)^{\chi} \right) [p^{\infty}]
					\longrightarrow
					\left( \widehat{E}^+(\mathfrak m_{m+1})^{\chi}
					/\omega _n \widehat{E}^+(\mathfrak m_{m+1})^{\chi} \right) [p^{\infty}]
			$$
			are multiplication-by-$p$ maps when $m$ is odd
			and identity maps when $m$ is even.
		We will calculate
			$\left(\widehat{E}^+(\mathfrak m_m)^{\chi}
			/\omega _n \widehat{E}^+(\mathfrak m_m)^{\chi} \right)[p^{\infty}]$
			for each $n$, $m$.
%		Remark that for the calculation of plus subgroups
%			$\left( \widehat{E}^+(\mathfrak m_m)^{\chi}
%			/\omega _n \widehat{E}^+(\mathfrak m_m)^{\chi} \right) (p)$
%			(resp. minus subgroups
%			$\left( \widehat{E}^-(\mathfrak m_m)^{\chi}
%			/\omega _n \widehat{E}^-(\mathfrak m_m)^{\chi} \right) (p)$),
		Since $\widehat{E}^+(\mathfrak m_m)=\widehat{E}^+(\mathfrak m_{m-1})$ if $m$ is odd,
			we may assume $m$ is even.

		We consider the case $n$ is even.
		By Proposition \ref{galois module structure of plus minus subgroups} we have
%%%%%%%%%%%%%%%%%%%%%%%%%%%%%%%%%%%%%%%%%%%%%%%%%%%%%%%%%
%%%%		ingredient 2 in pre-key proposition		%%%%%
%%%%%%%%%%%%%%%%%%%%%%%%%%%%%%%%%%%%%%%%%%%%%%%%%%%%%%%%%
		\begin{eqnarray}\label{ingredient 2 in pre-key proposition}
			\widehat{E}^+ (\mathfrak m_m)^{\chi}
			/\omega _n \widehat{E}^+ (\mathfrak m_m)^{\chi} \cong
			\frac{\mathbb Z_p[G_{-1}][X] \oplus \mathbb Z_p[G_{-1}]}
			{\left\langle (\widetilde{\omega} _m^+,-(\varphi + \varphi ^{-1})),
			(\omega _n,0)\right\rangle}.
		\end{eqnarray}
		
		We can show that
%%%%%%%%%%%%%%%%%%%%%%%%%%%%%%%%%%%%%%%%%%%%%%%%%%%%%
%%%%		torsion parts of coinvariants		%%%%%
%%%%%%%%%%%%%%%%%%%%%%%%%%%%%%%%%%%%%%%%%%%%%%%%%%%%%
		\begin{eqnarray}\label{torsion parts of coinvariants}
			\frac{\mathbb Z_p[G_{-1}][X] \oplus \mathbb Z_p[G_{-1}]}
				{\left\langle (\widetilde{\omega} _m^+,-(\varphi + \varphi ^{-1})),
				(\omega _n,0)\right\rangle}[p^{\infty}]
%			&& \cong \frac{\omega _n^+ \mathbb Z_p[G_{-1}][X]}
%				{\langle \omega _m^+, \omega _n\rangle}
%				\oplus \frac{\Ann_{\mathbb Z_p[G_{-1}]}(\varphi + \varphi ^{-1})}
%				{\langle p^{\frac{m-n}{2}}\rangle} \nonumber \\[1mm]
			\cong \frac{\mathbb Z_p[G_{-1}][X]}
				{\langle p^{\frac{m-n}{2}}, \widetilde{\omega}_n^- \rangle}
				\oplus \frac{\Ann_{\mathbb Z_p[G_{-1}]}(\varphi + \varphi ^{-1})}
				{\langle p^{\frac{m-n}{2}}\rangle}.
		\end{eqnarray}
		Indeed,
			since we have $\omega _m^+ \equiv p^{\frac{m-n}{2}} \omega _n^+$ (mod $\omega _n$)
			and $\omega _n = \widetilde{\omega }_n^- \omega _n^+$,
			there is an exact sequence
%%%%%%%%%%%%%%%%%%%%%%%%%%%%%%%%%%%%%%%%%%%%%%%%%%%%%%%%%
%%%%		decomposition into torsion and free		%%%%%
%%%%%%%%%%%%%%%%%%%%%%%%%%%%%%%%%%%%%%%%%%%%%%%%%%%%%%%%%
		\begin{eqnarray}\label{decomposition into torsion and free}
			0 & \longrightarrow &
				\frac{\mathbb Z_p[G_{-1}][X]}
				{\langle p^{\frac{m-n}{2}}, \widetilde{\omega} _n^- \rangle}
				\oplus \frac{\Ann_{\mathbb Z_p[G_{-1}]}(\varphi + \varphi ^{-1})}
				{\langle p^{\frac{m-n}{2}}\rangle}\nonumber \\
			& \longrightarrow &
				\frac{\mathbb Z_p[G_{-1}][X] \oplus \mathbb Z_p[G_{-1}]}
				{\left\langle (\widetilde{\omega} _m^+,-(\varphi + \varphi ^{-1})),
				(\omega _n,0)\right\rangle} \nonumber \\[1mm]
			& \longrightarrow &
				\frac{\mathbb Z_p[G_{-1}][X] \oplus \mathbb  Z_p[G_{-1}]}
				{\langle (\widetilde{\omega} _m^+,-(\varphi + \varphi ^{-1})),
				(\omega _n^+,0), (\alpha \widetilde{\omega}_n^+,0)\rangle}
			\longrightarrow 0,
		\end{eqnarray}
		where the first map is $(x,y) \mapsto (x \omega _n^+ + y \widetilde{\omega}_n^+,0)$
			and $\alpha$ is a generator of the $\mathbb Z_p[G_{-1}]$-module
			$\Ann _{\mathbb Z_p[G_{-1}]} (\varphi + \varphi ^{-1})$
			(cf. Lemma \ref{on the annihilator of varphi + varphi^{-1}}).
		There is an another exact sequence
%%%%%%%%%%%%%%%%%%%%%%%%%%%%%%%%%%%%%%%%%%%%%%%%%%%%%%%%%
%%%%%		decomposition into free and free 1		%%%%%
%%%%%%%%%%%%%%%%%%%%%%%%%%%%%%%%%%%%%%%%%%%%%%%%%%%%%%%%%
		\begin{eqnarray}\label{decomposition into free and free 1}
			0 & \longrightarrow &
				\frac{\mathbb Z_p[G_{-1}][X]}{\langle \omega _n^+,
				\alpha \widetilde{\omega}_n^+\rangle} \nonumber \\
			& \longrightarrow &
				\frac{\mathbb Z_p[G_{-1}][X] \oplus \mathbb  Z_p[G_{-1}]}
				{\langle (\widetilde{\omega} _m^+,-(\varphi + \varphi ^{-1})),
				(\omega _n^+,0), (\alpha \widetilde{\omega}_n^+,0)\rangle} \nonumber \\
			& \longrightarrow &
				\frac{\mathbb Z_p[G_{-1}]}{\langle \varphi + \varphi ^{-1}\rangle}
				\longrightarrow 0,
		\end{eqnarray}
		whose leftmost and rightmost modules are both $\mathbb Z_p$-free,
			where the first map is $x \mapsto (x,0)$
			and the second map is $(x,y) \mapsto y$.
		Thus the rightmost module in (\ref{decomposition into torsion and free}),
			which is the same as the middle module in (\ref{decomposition into free and free 1}),
			is $\mathbb Z_p$-free.
%		Thus we have
%		\begin{eqnarray*}
%			&&\frac{\mathbb Z_p[G_{-1}][X] \oplus \mathbb Z_p[G_{-1}]}
%				{\left\langle (\widetilde{\omega} _m^+,-(\varphi + \varphi ^{-1})),
%				(\omega _n,0)\right\rangle}(p)\\
%			&&= \left( \frac{\omega _n^+ \mathbb Z_p[G_{-1}][X]}
%				{\langle \omega _m^+, \omega _n \rangle}
%				\oplus \frac{\Ann_{\mathbb Z_p[G_{-1}]}(\varphi + \varphi ^{-1})}
%				{\langle p^{\frac{m-n}{2}}\rangle} \right) (p).
%		\end{eqnarray*}
%		Since $\omega _m^+ \equiv p^{\frac{m-n}{2}} \omega _n^+$
%			(mod $\omega _n$) and
%			$\omega _n = \widetilde{\omega}_n^- \omega _n^+$,
		Our claim (\ref{torsion parts of coinvariants}) follows from this.

%%%%%%%%%%%%%%%%%%%%%%%%%%%%%%%%
%%%%		conclusion		%%%%
%%%%%%%%%%%%%%%%%%%%%%%%%%%%%%%%
		By (\ref{ingredient 1 in pre-key proposition}), (\ref{ingredient 2 in pre-key proposition}),
		and (\ref{torsion parts of coinvariants}), we get
		\begin{eqnarray*}
			\left(\widehat{E}^+ (\mathfrak m_{\infty})^{\chi} \right) _{\Gamma _n} [p^{\infty}]
			& \cong & \underset{\substack{\longrightarrow \\ m \geq n}}{\lim}\
				\left( \widehat{E}^{\pm}(\mathfrak m_m)^{\chi}
				/\omega _n \widehat{E}^{\pm}(\mathfrak m_m)^{\chi} \right) [p^{\infty}] \\[1mm]
			& \cong & \underset{\substack{\longrightarrow \\ m \geq n}}{\lim}\
				\frac{\mathbb Z_p[G_{-1}][X] \oplus \mathbb Z_p[G_{-1}]}
				{\left\langle (\widetilde{\omega} _m^+,-(\varphi + \varphi ^{-1})),
				(\omega _n,0)\right\rangle}[p^{\infty}] \\[1mm]
			& \cong & \underset{\substack{\longrightarrow \\ m \geq n}}{\lim}\
				\frac{\mathbb Z_p[G_{-1}][X]}
				{\langle p^{\frac{m-n}{2}}, \widetilde{\omega}_n^- \rangle}
				\oplus \frac{\Ann_{\mathbb Z_p[G_{-1}]}
				(\varphi + \varphi ^{-1})}{\langle p^{\frac{m-n}{2}}\rangle} \\[1mm]
			& \cong & \left( \frac{\mathbb Z_p[G_{-1}][X]}
				{\langle \widetilde{\omega}_n^- \rangle}
				\oplus \Ann _{\mathbb Z_p[G_{-1}]}(\varphi + \varphi ^{-1})\right)
				\otimes \mathbb Q_p/\mathbb Z_p
		\end{eqnarray*}
		when $n$ is even.
		By replacing $p^{\frac{m-n}{2}}$ with $p^{\frac{m-(n-1)}{2}}$ in the above discussion,
			we get the statement also in the case where $n$ is odd.
		
		From this description, we see that
			$\left((\widehat{E}^+ (\mathfrak m_{\infty})^{\chi} ) _{\Gamma _n} [p^{\infty}]\right)^{\vee}$
			is $\mathbb Z_p$-free and
		\begin{eqnarray*}
			&& \rank _{\mathbb Z_p}
			\left((\widehat{E}^+ (\mathfrak m_{\infty})^{\chi} ) _{\Gamma _n} [p^{\infty}] \right)^{\vee} \\
			&& = \rank _{\mathbb Z_p} \left( \frac{\mathbb Z_p[G_{-1}][X]}
				{\langle \widetilde{\omega}_n^-\rangle} \right)
				+ \rank _{\mathbb Z_p} \left( \Ann _{\mathbb Z_p[G_{-1}]}(\varphi + \varphi ^{-1}) \right) \\
			&& = d q_n^- + \delta.
		\end{eqnarray*}
	\end{proof}

%%%%%%%%%%%%%%%%%%%%%%%%%%%%%%%%%%%%%%%%%%%%%%%%%%%%%%%%%%%%%%%%%
%%%%		[cor]	Z_p-ranks of Gamma_n coinvariants		%%%%%
%%%%%%%%%%%%%%%%%%%%%%%%%%%%%%%%%%%%%%%%%%%%%%%%%%%%%%%%%%%%%%%%%
	\begin{cor}\label{Z_p-ranks of Gamma_n coinvariants}
		Let $\chi :\Delta \rightarrow \mathbb Z_p^{\times}$ be a character.
		Then the $\Gamma_n$-coinvariants
			$( ( \widehat{E}^{\pm}(\mathfrak m_{\infty})^{\chi}
			\otimes \mathbb Q_p/\mathbb Z_p ) ^{\vee})_{\Gamma_n}$
			are free $\mathbb Z_p$-modules for all $n$, and we have
		\begin{eqnarray*}
			&& \rank _{\mathbb Z_p} \left(\left(
				\widehat{E}^+(\mathfrak m_{\infty})^{\chi}
				\otimes \mathbb Q_p/\mathbb Z_p
				\right) ^{\vee}\right)_{\Gamma _n}
			= dp^n + \delta ,\\
			&& \rank _{\mathbb Z_p} \left(\left(
				\widehat{E}^-(\mathfrak m_{\infty})^{\chi}
				\otimes \mathbb Q_p/\mathbb Z_p
				\right) ^{\vee}\right)_{\Gamma _n}
			= dp^n.
		\end{eqnarray*}
	\end{cor}
	
	\begin{proof}
		It follows from Corollary \ref{Z_p-ranks of norm subgroups},
			Proposition \ref{pre-key proposition}
			and the exact sequence (\ref{invariant coinvariant exact seqn's concerning local conditions}).
	\end{proof}

%%%%%%%%%%%%%%%%%%%%%%%%%%%%%%%%%%%%%%%%%%%%%%%%%%%%%%%%%%%%%%%%%%%%%%%%%
%%%%		[rem]	Remark on Z_p-ranks of Gamma_n coinvariants		%%%%%
%%%%%%%%%%%%%%%%%%%%%%%%%%%%%%%%%%%%%%%%%%%%%%%%%%%%%%%%%%%%%%%%%%%%%%%%%
	\begin{rem}\label{Remark on Z_p-ranks of Gamma_n coinvariants}
		\rm{From Corollary \ref{Z_p-ranks of Gamma_n coinvariants},
			we find that $(\widehat{E}^+ (\mathfrak m_{\infty})^{\chi} \otimes \mathbb Q_p/\mathbb Z_p)^{\vee}$ is not
			a free $\Lambda$-module
			in the case when $\delta =2$,
			i.e. the case when $d \equiv 0$ (mod $4$) and $\chi = \mathbf 1$.
		Indeed,
			the $\mathbb Z_p$-rank of the $\Gamma _n$-coinvariant
			of a free $\Lambda$-module
			is divisible by $p^n$ for each $n$.
		On the other hand,
			the $\mathbb Z_p$-rank of the $\Gamma_n$-coinvariant
			of $(\widehat{E}^+ (\mathfrak m_{\infty})^{\chi} \otimes \mathbb Q_p/\mathbb Z_p)^{\vee}$
			is not divisible by $p^n$ as in the corollary.
		}
	\end{rem}

%%%%%%%%%%%%%%%%%%%%%%%%%%%%%%%%%%%%%%%%%%%%
%%%%		[prop]	key proposition		%%%%
%%%%%%%%%%%%%%%%%%%%%%%%%%%%%%%%%%%%%%%%%%%%
	\begin{prop}\label{key proposition}
		Let $\chi :\Delta \rightarrow \mathbb Z_p^{\times}$ be a character.
		There exist injective homomorphisms of $\Lambda$-modules
		\begin{eqnarray*}
			&&\left( \widehat{E}^+(\mathfrak m_{\infty})^{\chi}
				\otimes \mathbb Q_p/\mathbb Z_p
				\right)^{\vee}
			\longrightarrow \Lambda ^{\oplus d} \oplus (\Lambda /X)^{\oplus \delta}, \\
			&&\left( \widehat{E}^-(\mathfrak m_{\infty})^{\chi}
				\otimes \mathbb Q_p/\mathbb Z_p
				\right)^{\vee}
			\longrightarrow \Lambda ^{\oplus d}
		\end{eqnarray*}
		with finite cokernels.
	\end{prop}
	
	\begin{proof}
		We prove the claim for $(\widehat{E}^+(\mathfrak m_{\infty})^{\chi}
			\otimes \mathbb Q_p/\mathbb Z_p)^{\vee}$.
		We can prove the rest of the claims similarly.
		
		We first note that
			$( \widehat{E}^+(\mathfrak m_{\infty})^{\chi}
			\otimes \mathbb Q_p/\mathbb Z_p ) ^{\vee}$
			has no nontrivial finite $\Lambda$-submodule
			since its $\Gamma_n$-coinvariants are free $\mathbb Z_p$-modules for all $n \geq 0$
			(see Corollary \ref{Z_p-ranks of Gamma_n coinvariants}).
		Thus by the structure theorem for $\Lambda$-modules,
			there exist irreducible distinguished polynomials $f_j$,
			nonnegative integers $r$, $s$, $t$, $m_i$, $n_j$,
			and an injective homomorphism
		\begin{eqnarray*}
			f:\left( \widehat{E}^+(\mathfrak m_{\infty})^{\chi}
			\otimes \mathbb Q_p/\mathbb Z_p \right)^{\vee}
			\longrightarrow
			\Lambda ^{\oplus r}
			\oplus \bigoplus_{i=1}^s \Lambda /p^{m_i}
			\oplus \bigoplus_{j=1}^t \Lambda /f_j^{n_j}=: \mathcal E
		\end{eqnarray*}
		with finite cokernel $Z$.
		
		We show that
		\begin{eqnarray*}
		\left\{
		\begin{array}{ll}
			r=d,\\[1mm]
			s=0\ (\text{in other words } m_i =0 \text{ for all }i),\\[1mm]
			t=\left\{
			\begin{array}{ll}
				0 & \text{if } d \nequiv 0 \text{ (mod }4) \text{ or } \chi \neq \mathbf 1,\\[1mm]
				2 & \text{otherwise},\ \text{and}
			\end{array}\right.\\[1mm]
			(f_1^{n_1}, \ldots ,f_t^{n_t})=  (X,X) \text{ if } t=2.
		\end{array}\right.
		\end{eqnarray*}
		
		From the exact sequence
			$0 \rightarrow \left( \widehat{E}^+(\mathfrak m_{\infty})^{\chi}
			\otimes \mathbb Q_p/\mathbb Z_p \right)^{\vee}
			\overset{f}{\rightarrow} \mathcal E
			\rightarrow Z
			\rightarrow 0$,
			we get the $\Gamma_n$-invariant-coinvariant exact sequences
%%%%%%%%%%%%%%%%%%%%%%%%%%%%%%%%%%%%%%%%%%%%%%%%%%%%%%%%%%%%%%%%%%%%%%%%
%%%%		invariant-coinvariant from the structure theorem		%%%%
%%%%%%%%%%%%%%%%%%%%%%%%%%%%%%%%%%%%%%%%%%%%%%%%%%%%%%%%%%%%%%%%%%%%%%%%
		\begin{eqnarray}\label{invariant-coinvariant from the structure theorem}
			Z^{\Gamma _n}
			& \longrightarrow &
				\left( ( \widehat{E}^+(\mathfrak m_{\infty})^{\chi}
				\otimes \mathbb Q_p/\mathbb Z_p )^{\vee}
				\right)_{\Gamma _n}\nonumber \\[3mm]
			& \longrightarrow &
				\mathcal E/\omega _n \mathcal E
				\longrightarrow Z/\omega _n Z \longrightarrow 0
		\end{eqnarray}
		for all $n$.
		Note that, the first maps in
			(\ref{invariant-coinvariant from the structure theorem})
			are $0$-maps for all $n$,
			since $( ( \widehat{E}^+(\mathfrak m_{\infty})^{\chi}
			\otimes \mathbb Q_p/\mathbb Z_p )^{\vee}
			)_{\Gamma _n}$ is $\mathbb Z_p$-free.
		Then we see that $m_i=0$ and $f_j^{n_j}|\omega _n$ (and $n_j \leq 1$)
			for all sufficiently large $n$
			since $Z/\omega _n Z$ is bounded as $n \rightarrow \infty$.
		Thus we get $s=0$ here.
		We now have
%%%%%%%%%%%%%%%%%%%%%%%%%%%%%%%%%%%%%%%%%%%%
%%%%		comparing the Z_p-ranks		%%%%
%%%%%%%%%%%%%%%%%%%%%%%%%%%%%%%%%%%%%%%%%%%%
		\begin{eqnarray}\label{comparing the Z_p-ranks}
			dp^n + \delta
			& = & \rank _{\mathbb Z_p}
				\left( (\widehat{E}^+ (\mathfrak m_{\infty})^{\chi} 
				\otimes \mathbb Q_p/\mathbb Z_p)^{\vee} \right) _{\Gamma _n}\nonumber \\
			& = & \rank _{\mathbb Z_p} ( \mathcal E/\omega _n \mathcal E )
				= rp^n + \sum _{j=1}^t n_j \deg f_j
		\end{eqnarray}
		for all sufficiently large $n$.
		Thus we get $r=d$.
		
		In the case when $d \nequiv 0$ (mod $4$) or $\chi \neq \mathbf 1$,
			we get $0 = \sum _{j=1}^t n_j \deg f_j$ from the above discussion.
		Thus we get $n_j=0$ which is the desired result, i.e. $t=0$.
		
		We finally consider the case when $d \equiv 0$ (mod $4$) and $\chi = \mathbf 1$.
		We may assume $n_j =1$ for all $j$.
		In this case, we have
		\begin{eqnarray}\label{condition 1}
			&&2 = \sum _{j=1}^t \deg f_j\\
			\label{condition 2}
			&&f_j | \omega _n (= (1+X)^{p^n}-1)
		\end{eqnarray}
		for all sufficiently large $n$.
		We narrow down the possible combinations of $(t,(f_1,\ldots ,f_t))$
			satisfying these two conditions (\ref{condition 1}) and (\ref{condition 2}).
		If $p \geq 5$, there is a unique combination
			$(t,(f_1 ,\ldots ,f_t)) = (2,(X,X))$,
			since $\deg f_j \leq 2$.
		If $p=3$, since $\omega _1=X(X^2+3X+3)$,
			there are two possible combinations
			$(t,(f_1 ,\ldots ,f_t))
			= (2,(X,X))$, $(1, (X^2+3X+3))$.
		By showing that the last combination is impossible, we complete the proof.
		Indeed, we have $\rank _{\mathbb Z_p} \mathcal E/\omega _0 \mathcal E=d+1$
			with the combination $(t,(f_1 ,\ldots ,f_t)) =(1, (X^2+3X+3))$.
		On the other hand, 
			from the exact sequence
			(\ref{invariant-coinvariant from the structure theorem})
			for $n=0$,
			we must have $\rank _{\mathbb Z_p}\mathcal E/\omega _0 \mathcal E = d+2$
			and thus we get the desired conclusion.
	\end{proof}
	
%%%%%%%%%%%%%%%%%%%%%%%%%%%%%%%%%%%%%%%%%%%%%%%%%%%%%%%%%%%%%%%%%
%%%%%		[prop] conclusion of the key  proposition		%%%%%
%%%%%%%%%%%%%%%%%%%%%%%%%%%%%%%%%%%%%%%%%%%%%%%%%%%%%%%%%%%%%%%%%
%	\begin{prop}\label{conclusion of the key proposition}
%		Let $\chi : \Delta \rightarrow \mathbb Z_p^{\times}$ be a character.
%		Then $(\widehat{E}^{\pm}(\mathfrak m_{\infty})^{\chi}
%			\otimes \mathbb Q_p/\mathbb Z_p)^{\vee}$
%			have no nontrivial finite $\Lambda$-submodule
%			and their $\Lambda$-ranks are both $d$.
%	\end{prop}
%	
%	\begin{proof}
%		This follows from Proposition \ref{pre-key proposition}
%			and Proposition \ref{key proposition}.
%	\end{proof}
	We now get the following proposition
		which is an important ingredient
		for the proof of Proposition \ref{Lambda module structure of H^1/E^{pm}}.
	
%%%%%%%%%%%%%%%%%%%%%%%%%%%%%%%%%%%%%%%%%%%%%%%%%%%%%%%%%%%%%%%%%%%%%%%%%%%%%%%%%%%%
%%%%		[prop]	structure of the plus and the minus local conditions		%%%%
%%%%%%%%%%%%%%%%%%%%%%%%%%%%%%%%%%%%%%%%%%%%%%%%%%%%%%%%%%%%%%%%%%%%%%%%%%%%%%%%%%%%
	\begin{prop}\label{structure of the plus and the minus local conditions}
		Let $\chi : \Delta \rightarrow \mathbb Z_p^{\times}$ be a character.
		Then $(\widehat{E}^{\pm}(\mathfrak m_{\infty})^{\chi} \otimes \mathbb Q_p/\mathbb Z_p)^{\vee}$
			has no nontrivial finite $\Lambda$-submodule
			and its $\Lambda$-rank is $d$.
	\end{prop}
	
	\begin{proof}
		This follows from Corollary \ref{Z_p-ranks of Gamma_n coinvariants}
			and Proposition \ref{key proposition}.
	\end{proof}

	In the rest of this subsection,
		we study the $\Lambda$-module
		$$\left( \frac{H^1(k_{\infty},E[p^{\infty}])}
		{\widehat{E}^{\pm}(\mathfrak m_{\infty}) \otimes \mathbb Q_p/\mathbb Z_p }\right)^{\vee}.$$
	
	We consider the following exact sequence;
%%%%%%%%%%%%%%%%%%%%%%%%%%%%%%%%%%%%%%%%%%%%%%%%%%%%
%%%%		exact sequence on H^1/E^{pm}		%%%%
%%%%%%%%%%%%%%%%%%%%%%%%%%%%%%%%%%%%%%%%%%%%%%%%%%%%
	\begin{eqnarray}\label{exact sequence on H^1/E^{pm}}
		0 \rightarrow
				\left( \frac{H^1(k_{\infty}, E[p^{\infty}])}
				{\widehat{E}^{\pm} (\mathfrak m_{\infty}) \otimes \mathbb Q_p/\mathbb Z_p}\right) ^{\vee}
			& \longrightarrow &
				\left( H^1(k_{\infty}, E[p^{\infty}]) \right)^{\vee} \nonumber \\
			& \longrightarrow &
				\left( \widehat{E}^{\pm} (\mathfrak m_{\infty}) \otimes \mathbb Q_p/\mathbb Z_p\right) ^{\vee}
			\rightarrow 0.
	\end{eqnarray}
	We studied the $\Lambda$-module structure of the rightmost module.
	We also know the $\Lambda$-module structure of the middle module
		by the following fact (Proposition \ref{Greenberg's lemma 2}).
	
%%%%%%%%%%%%%%%%%%%%%%%%%%%%%%%%%%%%%%%%%%%%%%%%
%%%%		[prop]	Greenberg's lemma 2		%%%%
%%%%%%%%%%%%%%%%%%%%%%%%%%%%%%%%%%%%%%%%%%%%%%%%
	\begin{prop}[Greenberg \cite{Gre89} \S 3 Corollary 2]\label{Greenberg's lemma 2}
		Let $K$ be a finite extension of $\mathbb Q_p$ and
			$K_{\infty}$ a $\mathbb Z_p$-extension of $K$.
		Put $\Lambda _K= \mathbb Z_p[[\Gal (K_{\infty}/K)]]$.
		If $E(K_{\infty})[p^{\infty}]=0$,
			then $H^1(K_{\infty}, E[p^{\infty}])^{\vee}$
			is a free $\Lambda_K$-module
			and its $\Lambda_K$-rank is $2[K:\mathbb Q_p]$;
		\begin{eqnarray*}
			H^1(K_{\infty},E[p^{\infty}])^{\vee} \cong \Lambda_K ^{\oplus 2[K:\mathbb Q_p]}.
		\end{eqnarray*}
	\end{prop}
	
	We can apply Proposition \ref{Greenberg's lemma 2}
		in our setting as $K=k_0$, $K_{\infty}=k_{\infty}$.
	Indeed we see that $E(k_{\infty})[p^{\infty}]=0$
		by Proposition \ref{E^ is torsion-free}.

%%%%%%%%%%%%%%%%%%%%%%%%%%%%%%%%%%%%%%%%%%%%%%%%%%%%%%%%%%%%%%%%%%%%%%%%
%%%%		[prop]	structure of the target ((key proposition))		%%%%
%%%%%%%%%%%%%%%%%%%%%%%%%%%%%%%%%%%%%%%%%%%%%%%%%%%%%%%%%%%%%%%%%%%%%%%%
%	\begin{prop}\label{structure of the target ((key proposition))}
%		$(\widehat{E}^{\pm}(F_{\infty ,v}) \otimes \mathbb Q_p/\mathbb Z_p)^{\vee}$
%			has no non-trivial finite $\Lambda$-submodule
%			and the $\Lambda$-rank is $[F_{0,v}:\mathbb Q_p]$.
%	\end{prop}
%	
%	\begin{proof}
%		This follows from Corollary \ref{Z_p-ranks of Gamma_n coinvariants}
%			and Proposition \ref{key proposition}.
%	\end{proof}
	
	Here we recall the following useful lemma on equivalent conditions
		on freeness of $\Lambda$-modules and on triviality of finite $\Lambda$-submodules.
%%%%%%%%%%%%%%%%%%%%%%%%%%%%%%%%%%%%%%%%%%%%%%%%%%%%%%%%%%%%%%%%%%%%%%%%%%%%%%%%%%%%%%%%%%%%%%%%%%%%%%%%
%%%%		[lemma]	equivalent conditions on freenes and triviality of finite lambda-submodules		%%%%
%%%%%%%%%%%%%%%%%%%%%%%%%%%%%%%%%%%%%%%%%%%%%%%%%%%%%%%%%%%%%%%%%%%%%%%%%%%%%%%%%%%%%%%%%%%%%%%%%%%%%%%%
	\begin{lemma}\label{equivalent conditions on freenes and triviality of finite lambda-submodules}
		Let $M$ be a finitely generated $\Lambda$-module.
		
		(1) $M$ is a free $\Lambda$-module
			if and only if $M^{\Gamma}=0$ and $M_{\Gamma}$ is a free $\mathbb Z_p$-module.
			
		(2) $M$ has no nontrivial finite $\Lambda$-submodule
			if and only if $M^{\Gamma}$ is a free $\mathbb Z_p$-module.
	\end{lemma}
	
	\begin{proof}
		See for example \cite{NSW} Proposition 5.3.19.
	\end{proof}
	
	Applying the following lemma to the exact sequence (\ref{exact sequence on H^1/E^{pm}}),
	we can now determine the $\Lambda$-module structure of
		$( H^1( k_{\infty}, E[p^{\infty}] )
		/ (\widehat{E}^{\pm}(k_{\infty}) \otimes \mathbb Q_p/\mathbb Z_p))^{\vee}$.
	
%%%%%%%%%%%%%%%%%%%%%%%%%%%%%%%%%%%%%%%%%%%%%%%%%%%%%%%%%
%%%%		[lemma]	on freenes of the kernel		%%%%%
%%%%%%%%%%%%%%%%%%%%%%%%%%%%%%%%%%%%%%%%%%%%%%%%%%%%%%%%%
	\begin{lemma}\label{on freenes of the kernel}
		Let $f : M \rightarrow N$
			be a surjective homomorphism of $\Lambda$-modules.
		Suppose that $M$ is a free $\Lambda$-module of $\Lambda$-rank $r$,
			and that $N$ is $\Lambda$-module of $\Lambda$-rank $s$
			which has no non-trivial finite $\Lambda$-submodule.
		Then its kernel $\Ker f$ is a free $\Lambda$-module of rank $r-s$.
	\end{lemma}
	
	\begin{proof}
		We put $M_0 := \Ker f$.
		Then by taking the invariant-coinvariant exact sequence,
			we have
		$$
			0 \longrightarrow M_0^\Gamma \longrightarrow M^\Gamma \longrightarrow N^\Gamma
			\longrightarrow M_{0,\Gamma} \longrightarrow M_\Gamma .
		$$
		Since $M$ is a free $\Lambda$-module,
			we have $M^\Gamma = 0$
			and $M_\Gamma$ is a free $\mathbb Z_p$-module
			by Lemma \ref{equivalent conditions on freenes and triviality of finite lambda-submodules}.
		Since $N$ has no non-trivial finite $\Lambda$-submodule,
			we see that $N^\Gamma$ is a free $\mathbb Z_p$-module
			by Lemma \ref{equivalent conditions on freenes and triviality of finite lambda-submodules}.
		Hence we have $M_0^\Gamma = 0$
			and $M_{0, \Gamma}$ is a free $\mathbb Z_p$-module.
		Thus $M_0$ is a free $\Lambda$-module
			again by Lemma \ref{equivalent conditions on freenes and triviality of finite lambda-submodules}.
		It is easy to see that the $\Lambda$-rank of $M_0$ is $r-s$.
	\end{proof}

%%%%%%%%%%%%%%%%%%%%%%%%%%%%%%%%%%%%%%%%%%%%%%%%%%%%%%%%%%%%%%%%%%%%
%%%%		[prop]	Lambda module structure of H^1/E^{pm}		%%%%
%%%%%%%%%%%%%%%%%%%%%%%%%%%%%%%%%%%%%%%%%%%%%%%%%%%%%%%%$%%%%%%%%%%%
	\begin{prop}\label{Lambda module structure of H^1/E^{pm}}
		We have
		\begin{eqnarray*}
			\left( \frac{H^1(k_{\infty},E[p^{\infty}])}{E^{\pm}(k_{\infty})
			\otimes \mathbb Q_p/\mathbb Z_p} \right) ^{\vee}
			\cong \Lambda ^{\oplus [k_0:\mathbb Q_p]}.
		\end{eqnarray*}
	\end{prop}
	
	\begin{proof}
		It follows from Proposition \ref{structure of the plus and the minus local conditions},
			Proposition \ref{Greenberg's lemma 2},
			and Lemma \ref{on freenes of the kernel}
			for the exact sequence (\ref{exact sequence on H^1/E^{pm}}).
	\end{proof}

%%%%%%%%%%%%%%%%%%%%%%%%%%%%%%%%%%%%%%%%%%%%%%%%%%%%%%%%%%%%%%%%
%%%%%%%%%%%%%%%%%%%%%%%%%%%%%%%%%%%%%%%%%%%%%%%%%%%%%%%%%%%%%%%%
%%%%	§3.4 §3.4 §3.4 	%%%%%%%%	§3.4 §3.4 §3.4	%%%%
%%%%%%%%%%%%%%%%%%%%%%%%%%%%%%%%%%%%%%%%%%%%%%%%%%%%%%%%%%%%%%%%
%%%%%%%%%%%%%%%%%%%%%%%%%%%%%%%%%%%%%%%%%%%%%%%%%%%%%%%%%%%%%%%%
\subsection{More on the plus and the minus local conditions}
	
	The discussion in the previous subsections is enough to prove our main theorem.
	In this subsection,
		we proceed to determine the explicit structure of the $\Lambda$-module
		$\left( \widehat{E}^{\pm}(\mathfrak m_{\infty})^{\chi}
		\otimes \mathbb Q_p/\mathbb Z_p \right) ^{\vee}$.
	For that purpose, we show the following lemma.

%%%%%%%%%%%%%%%%%%%%%%%%%%%%%%%%%%%%%%%%%%%%%%%%
%%%%		[lemma]	supplementary lemma		%%%%
%%%%%%%%%%%%%%%%%%%%%%%%%%%%%%%%%%%%%%%%%%%%%%%%
	\begin{lemma}\label{supplementary lemma}
		Let $f:M \rightarrow N$ be an injective homomorphism of $\Lambda$-modules
			with finite cokernel.
		Suppose that $M/\omega _n M$ is $\mathbb Z_p$-free
			and $N_{\Lambda \text{-tors}} = \{ x \in N | \omega _n x =0\}$
			for all sufficiently large $n$.
		Then $f$ induces an isomorphism
		\begin{eqnarray*}
			\overline{f} :M/M_{\Lambda \text{-tors}} \overset{\simeq}{\longrightarrow}
			N/N_{\Lambda \text{-tors}}
		\end{eqnarray*}
		for all sufficiently large $n$.
	\end{lemma}
	
	\begin{proof}
		We regard $M$ as a $\Lambda$-submodule of $N$ by $f$.
		Since $N/M$ is finite, we see that
			$\omega _n N \subset M$ for all sufficiently large $n$.
		We thus have
		\begin{eqnarray*}
			\Coker (\overline{f})
			&=& N/(M+N_{\Lambda \text{-tors}}) \\
			&\overset{\substack{\times \omega _n \\ \simeq}}{\longrightarrow}&
				\omega _n N/\omega _n M \\
			& \hookrightarrow & M/\omega _n M
		\end{eqnarray*}
		for all sufficiently large $n$.
		Since $M/\omega _n M$ is $\mathbb Z_p$-free for all sufficiently large $n$,
			we get $M/M_{\Lambda \text{-tors}} \cong N/N_{\Lambda \text{-tors}}$
			for such $n$.
	\end{proof}

%%%%%%%%%%%%%%%%%%%%%%%%%%%%%%%%%%%%%%%%%%%%%%%%%%%%
%%%%		[thm]	supplementary theorem		%%%%
%%%%%%%%%%%%%%%%%%%%%%%%%%%%%%%%%%%%%%%%%%%%%%%%%%%%
	\begin{thm}\label{supplementary theorem}
		Let $\chi :\Delta \rightarrow \mathbb Z_p^{\times}$ be a character.
		We have
		\begin{eqnarray*}
			\left( \widehat{E}^+(\mathfrak m_{\infty})^{\chi}
				\otimes \mathbb Q_p/\mathbb Z_p \right) ^{\vee}
				& \cong &
				\Lambda ^{\oplus d} \oplus
				(\Lambda /X)^{\oplus \delta},\\
			\left( \widehat{E}^-(\mathfrak m_{\infty})^{\chi}
				\otimes \mathbb Q_p/\mathbb Z_p \right) ^{\vee}
				& \cong &
				\Lambda ^{\oplus d}.
		\end{eqnarray*}
	\end{thm}
	
	\begin{proof}
		We prove this theorem for
			$( \widehat{E}^+(\mathfrak m_{\infty})^{\chi}
			\otimes \mathbb Q_p/\mathbb Z_p ) ^{\vee}$.
		We can prove the rest of the claim similarly.
		
		Let $M = \left( \widehat{E}^+ (\mathfrak m_{\infty})^{\chi}
			\otimes \mathbb Q_p/\mathbb Z_p \right) ^{\vee}$,
			$N = \Lambda ^{\oplus d} \oplus (\Lambda /X)^{\oplus \delta}$,
			and $f$ be the map obtained in Proposition \ref{key proposition}.	
		We consider the following commutative diagram;
		\[
		\xymatrix{
			0 \ar[r] & M_{\Lambda \text{-tors}} \ar[r] \ar@{^{(}->}[d]^{f_0}
				& M \ar[r] \ar@{^{(}->}[d]^{f}
				& M/M_{\Lambda \text{-tors}} \ar[r] \ar@{^{(}->}[d]^{\overline{f}}
				& 0 \\
			0 \ar[r] & (\Lambda /X)^{\oplus \delta} \ar[r]
				& \Lambda ^{\oplus d} \oplus (\Lambda /X)^{\oplus \delta} \ar[r]
				& \Lambda ^{\oplus d} \ar[r]
				& 0.
		}
		\]
		Here we note that $\overline{f}$ is actually injective.
		Indeed, we see that $\Ker (\overline{f})$ is a $\Lambda$-torsion submodule
			of a $\Lambda$-torsion-free module $M/M_{\Lambda \text{-tors}}$,
			since the $\Lambda$-rank of $M/M_{\Lambda \text{-tors}}$ is $d$
			and $\Coker (\overline{f})$ is finite.

		We see that $M_{\Lambda \text{-tors}} \cong (\Lambda /X)^{\oplus \delta}$,
		since $\Coker (f_0)$ is finite.
		On the other hand,
		the assumptions in Lemma \ref{supplementary lemma} are satisfied
			(see Proposition \ref{key proposition},
			Corollary \ref{Z_p-ranks of Gamma_n coinvariants}).
		Thus we have $M/M_{\Lambda \text{-tors}} \cong \Lambda ^{\oplus d}$
			by Lemma \ref{supplementary lemma}.
		Therefore, the above horizontal exact sequence
		splits and thus we get
		\begin{eqnarray*}
			M \cong M/M_{\Lambda \text{-tors}} \oplus M_{\Lambda \text{-tors}}
			\cong \Lambda ^{\oplus d} \oplus (\Lambda /X)^{\oplus \delta}.
		\end{eqnarray*}
	\end{proof}

%%%%%%%%%%%%%%%%%%%%%%%%%%%%%%%%%%%%%%%%%%%%%%%%%%%%%%%%%%%%%%%%%%%%%%%%%%%%%%%%%%%%%%%%%%%%%%%%%%%%%%%%%%%%
%%%%%%%%%%%%%%%%%%%%%%%%%%%%%%%%%%%%%%%%%%%%%%%%%%%%%%%%%%%%%%%%%%%%%%%%%%%%%%%%%%%%%%%%%%%%%%%%%%%%%%%%%%%%
%%%%	§4 §4 §4 §4 §4 	%%%%%%%%	§4 §4 §4 §4 §4 	%%%%%%%%	§4 §4 §4 §4 §4 	%%%%
%%%%%%%%%%%%%%%%%%%%%%%%%%%%%%%%%%%%%%%%%%%%%%%%%%%%%%%%%%%%%%%%%%%%%%%%%%%%%%%%%%%%%%%%%%%%%%%%%%%%%%%%%%%%
%%%%%%%%%%%%%%%%%%%%%%%%%%%%%%%%%%%%%%%%%%%%%%%%%%%%%%%%%%%%%%%%%%%%%%%%%%%%%%%%%%%%%%%%%%%%%%%%%%%%%%%%%%%%
\section{Finite $\Lambda$-submodules}
	
	We use the notations and the assumptions introduced in Section 2.
	
	Let $\Gamma = \Gal (F_{\infty}/F_0)$ and $\Lambda = \mathbb Z_p[[\Gamma]]$.
	We fix a topological generator $\gamma \in \Gamma$.
	Then we identify the completed group ring $\mathbb Z_p[[\Gamma]]$
		with the ring of power series $\mathbb Z_p[[X]]$ by identifying $\gamma$ with $1+X$.
	
%%%%%%%%%%%%%%%%%%%%%%%%%%%%%%%%%%%%%%%%%%%%%%%%%%%%%%%%%%%%%%%%
%%%%%%%%%%%%%%%%%%%%%%%%%%%%%%%%%%%%%%%%%%%%%%%%%%%%%%%%%%%%%%%%
%%%%	§4.1 §4.1 §4.1 	%%%%%%%%	§4.1 §4.1 §4.1	%%%%
%%%%%%%%%%%%%%%%%%%%%%%%%%%%%%%%%%%%%%%%%%%%%%%%%%%%%%%%%%%%%%%%
%%%%%%%%%%%%%%%%%%%%%%%%%%%%%%%%%%%%%%%%%%%%%%%%%%%%%%%%%%%%%%%%
\subsection{Finite $\Lambda$-submodules of $\Sel (F_{\infty},E[p^{\infty}])^{\vee}$}
	
	In this subsection, we study finite $\Lambda$-submodules of
		the Pontryagin dual of the $p$-primary Selmer group.
	The aim of this subsection is to prove Theorem \ref{main theorem I}.
	
	The following proposition is due to Matsuno \cite[Theorem 2.4]{Matsuno03}
		(see also Hachimori--Matsuno \cite{Hachimori-Matsuno00}).

%%%%%%%%%%%%%%%%%%%%%%%%%%%%%%%%%%%%%%%%%%%%%%%%%%%%%%%%%%%%%%%%%%%%%%%%%%%%%%%%
%%%%		[prop]	key prop for triviality of finite submodule of Sel		%%%%
%%%%%%%%%%%%%%%%%%%%%%%%%%%%%%%%%%%%%%%%%%%%%%%%%%%%%%%%%%%%%%%%%%%%%%%%%%%%%%%%
	\begin{prop}\label{key prop for triviality of finite submodule of Sel}
		Let $K$ be a finite extension of $\mathbb Q$,
			$K_{\infty}/K$ a $\mathbb Z_p$-extension,
			$K_n$ its $n$-th layer,
			and $E$ an elliptic curve defined over $K$.
		Put $\Gamma _K = \Gal (K_{\infty}/K)$,
			and $\Lambda _K = \mathbb Z_p[[\Gamma _K]]$.
		Let $X_n$ be the kernel of the restriction map
		\begin{eqnarray*}
			\Sel (K_n, E[p^{\infty}]) \longrightarrow \Sel (K_{\infty}, E[p^{\infty}])
		\end{eqnarray*}
			and $X_{\infty}:= \varprojlim X_n$
			where the projective limit is taken with respect to the corestriction maps.
		
		Assume that the $\mathbb Z_p$-rank of $\Sel ( K_n, E[p^{\infty}] )^{\vee}$
			is bounded as $n \rightarrow \infty$.
		Then the maximal finite $\Lambda _K$-submodule of $\Sel ( K_{\infty}, E[p^{\infty}] )^{\vee}$
			is isomorphic to $X_{\infty}$.
		
		In particular if we further assume that $E(K)[p]=0$,
			then $\Sel ( K_{\infty}, E[p^{\infty}] )^{\vee}$
			has no nontrivial finite $\Lambda_K$-submodule.
	\end{prop}
	
%	\begin{proof}
%		The proof in \cite{Hachimori-Matsuno00} works as well
%			assuming the boundedness of the $\mathbb Z_p$-ranks of $\Sel ( K_n,E[p^{\infty}] )^{\vee}$
%			instead of assuming that $\Sel (K_{\infty},E[p^{\infty}])^{\vee}$ is $\Lambda_K$-torsion,
%			as Takeji pointed out in \cite{Takeji14}.
%	\end{proof}
	
	We check that we can apply the above proposition in our setting
		$K=F_0$, $K_{\infty}=F_{\infty}$.

%%%%%%%%%%%%%%%%%%%%%%%%%%%%%%%%%%%%%%%%%%%%%%%%%%%%%%%%%%%%%%%%%%%%%%%%%%%%%%%%%%%%
%%%%		[lemma]	some morphisms of control theorem type are injective		%%%%
%%%%%%%%%%%%%%%%%%%%%%%%%%%%%%%%%%%%%%%%%%%%%%%%%%%%%%%%%%%%%%%%%%%%%%%%%%%%%%%%%%%%
	\begin{lemma}\label{morphisms of control theorem type are injective}
		The morphisms $\Sel^{\pm}(F_n,E[p^{\infty}]) \rightarrow \Sel ^{\pm}(F_{\infty},E[p^{\infty}])$
			are injective for all $n \geq 0$.
	\end{lemma}
	
	\begin{proof}
		We can prove this by the same method as the proof of Lemma 9.1 in \cite{Kob03}.
	\end{proof}
	
	We assume from here that both $\Sel ^{\pm}(F_{\infty}, E[p^{\infty}])^{\vee}$ are $\Lambda$-torsion.
	We denote the Iwasawa $\lambda$-invariant of $\Sel^{\pm} (F_{\infty},E[p^{\infty}])^{\vee}$
		by $\lambda ^{\pm}$.
	
	Let
	\begin{eqnarray*}
		\Sel^1(F_n, E[p^{\infty}])
			:= \Ker \left( \Sel (F_n ,E[p^{\infty}])
			\longrightarrow \prod _{v \in S_{p,F}^{\rm ss}} \frac{H^1(F_{n,v},E[p^{\infty}])}
			{E(F_v) \otimes \mathbb Q_p/\mathbb Z_p}
			\right),
	\end{eqnarray*}
	where $S_{p,F}^{\rm ss}$ is the set of all primes of $F$ lying above $p$
		where $E$ has supersingular reduction.
	
	By the exact sequence (\ref{exact sequence in norm subgroups}), we have an exact sequence
	\begin{eqnarray*}
		0 \longrightarrow
		\frac{ H^1( F_{n,v}, E[p^{\infty}] ) }{ E (F_v) \otimes \mathbb Q_p/\mathbb Z_p } & \longrightarrow &
		\frac{ H^1( F_{n,v}, E[p^{\infty}] ) }{ E^+(F_{n,v}) \otimes \mathbb Q_p/\mathbb Z_p }
			\oplus \frac{ H^1( F_{n,v}, E[p^{\infty}] ) }{ E^-(F_{n,v}) \otimes \mathbb Q_p/\mathbb Z_p }\\[3mm]
		& \longrightarrow &
		\frac{ H^1( F_{n,v}, E[p^{\infty}] ) }{ E (F_{n,v}) \otimes \mathbb Q_p/\mathbb Z_p } \longrightarrow
		0
	\end{eqnarray*}
	for each $n$ and for each prime $v$ of $F$ lying above $p$.
	Thus, for each $n$, we get the following exact sequence
%%%%%%%%%%%%%%%%%%%%%%%%%%%%%%%%%%%%%%%%%%%%%%%%%%%%%%%%
%%%%		exact sequence 1 for main theorem I		%%%%
%%%%%%%%%%%%%%%%%%%%%%%%%%%%%%%%%%%%%%%%%%%%%%%%%%%%%%%%
	\begin{eqnarray}\label{exact sequence 1 for main theorem I}
		0 \longrightarrow
		\Sel ^1( F_n, E[p^{\infty}] ) & \overset{\iota}{\longrightarrow} &
		\Sel ^+ ( F_n, E[p^{\infty}] ) \oplus \Sel ^-( F_n, E[p^{\infty}] ) \nonumber \\[3mm]
		& \overset{\eta}{\longrightarrow} &
		\Sel ( F_n, E[p^{\infty}] )
	\end{eqnarray}
	where $\iota$ is the diagonal embedding by inclusions and $\eta$ is $(x,y) \mapsto x-y$.

%%%%%%%%%%%%%%%%%%%%%%%%%%%%%%%%%%%%%%%%%%%%%%%%%%%%%%%%
%%%%		[prop]	cokernel of eta is finite		%%%%
%%%%%%%%%%%%%%%%%%%%%%%%%%%%%%%%%%%%%%%%%%%%%%%%%%%%%%%%
	\begin{prop}\label{cokernel of eta is finite}
		The cokernel of $\eta$ in the exact sequence (\ref{exact sequence 1 for main theorem I}) is finite.
	\end{prop}
	
	\begin{proof}
		We can prove this by the same method as the proof of Lemma 10.1 in \cite{Kob03}.
	\end{proof}

%	By the following proposition, we complete the proof of Theorem \ref{main theorem I}.
%%%%%%%%%%%%%%%%%%%%%%%%%%%%%%%%%%%%%%%%%%%%%%%%%%%%%%%%%%%%%
%%%%		[prop]	Z_p-corank of Sel is bounded		%%%%%
%%%%%%%%%%%%%%%%%%%%%%%%%%%%%%%%%%%%%%%%%%%%%%%%%%%%%%%%%%%%%
	\begin{prop}\label{Z_p-corank of Sel is bounded}
		The $\mathbb Z_p$-rank of $\Sel ( F_n, E[p^{\infty}] )^{\vee}$ is bounded as $n \rightarrow \infty$.
		More precisely, we have
		\begin{eqnarray*}
			\rank _{\mathbb Z_p} \Sel ( F, E[p^{\infty}])^{\vee}
			+\rank _{\mathbb Z_p} \Sel ( F_n, E[p^{\infty}] )^{\vee}
			\leq \lambda ^+ + \lambda ^-
		\end{eqnarray*}
		for every $n$.
	\end{prop}
	
	\begin{proof}
		Since the restriction map $\Sel ( F, E[p^{\infty}] ) \rightarrow \Sel ^1( F_n, E[p^{\infty}] )$ is injective,
			we have
		\begin{eqnarray*}
			\rank _{\mathbb Z_p} \Sel ( F ,E[p^{\infty}] )^{\vee}
			\leq \rank _{\mathbb Z_p} \Sel ^1( F_n ,E[p^{\infty}] )^{\vee}.
		\end{eqnarray*}
		Hence by Lemma \ref{morphisms of control theorem type are injective}
			and Proposition \ref{cokernel of eta is finite}, we get
		\begin{eqnarray*}
			&& \rank _{\mathbb Z_p} \Sel ( F, E[p^{\infty}] )^{\vee}
				+ \rank _{\mathbb Z_p} \Sel ( F_n, E[p^{\infty}] )^{\vee}\\
			&& \leq \rank _{\mathbb Z_p} \Sel ^+ ( F_n, E[p^{\infty}] )^{\vee}
				+ \rank _{\mathbb Z_p} \Sel ^- ( F_n, E[p^{\infty}] )^{\vee}\\
			&& \leq \lambda ^+ + \lambda ^-
		\end{eqnarray*}
		for every $n$ from (\ref{exact sequence 1 for main theorem I}).
		The boundedness of the $\mathbb Z_p$-ranks follows from this immediately.
	\end{proof}
	
	From the above argument,
		we can prove the following theorem.
%%%%%%%%%%%%%%%%%%%%%%%%%%%%%%%%%%%%%%%%%%%%
%%%%		[thm]	main theorem I		%%%%
%%%%%%%%%%%%%%%%%%%%%%%%%%%%%%%%%%%%%%%%%%%%
	\begin{thm}\label{main theorem I}
		Assume that both $\Sel ^+(F_{\infty}, E[p^{\infty}])^{\vee}$
			and $\Sel ^-(F_{\infty}, E[p^{\infty}])^{\vee}$
			are $\Lambda$-torsion.
		Then $\Sel (F_{\infty},E[p^{\infty}])^{\vee}$
			has no nontrivial finite $\Lambda$-submodule.
	\end{thm}
	
	\begin{proof}
		The $\mathbb Z_p$-rank of $\Sel (F_n,E[p^{\infty}])^{\vee}$ is bounded as $n \rightarrow \infty$
			(cf. Proposition \ref{Z_p-corank of Sel is bounded}).
		Further, we have $E(F_0)[p]=0$ by Proposition \ref{E^ is torsion-free}.
		Thus we can apply Proposition \ref{key prop for triviality of finite submodule of Sel}
			and get the desired result.
	\end{proof}

%%%%%%%%%%%%%%%%%%%%%%%%%%%%%%%%%%%%%%%%%%%%%%%%%%%%%%%%%%%%%%%%
%%%%%%%%%%%%%%%%%%%%%%%%%%%%%%%%%%%%%%%%%%%%%%%%%%%%%%%%%%%%%%%%
%%%%	§4.2 §4.2 §4.2 	%%%%%%%%	§4.2 §4.2 §4.2	%%%%
%%%%%%%%%%%%%%%%%%%%%%%%%%%%%%%%%%%%%%%%%%%%%%%%%%%%%%%%%%%%%%%%
%%%%%%%%%%%%%%%%%%%%%%%%%%%%%%%%%%%%%%%%%%%%%%%%%%%%%%%%%%%%%%%%
\subsection{Finite $\Lambda$-submodules of $\Sel^{\pm} (F_{\infty},E[p^{\infty}])^{\vee}$}
	
	Finally, we study finite $\Lambda$-submodules of
		the Pontryagin duals of the plus and the minus Selmer groups.
	The aim of this subsection is to prove our main theorem (Theorem \ref{main theorem II}).
	
	We prove that the triviality of finite $\Lambda$-submodules of $\Sel (F_{\infty},E[p^{\infty}]) ^{\vee}$
		is inherited to that of $\Sel^{\pm}(F_{\infty},E[p^{\infty}])^{\vee}$.

%	The proof of Theorem \ref{main theorem II} will fill the rest of this section.
	Let us consider the following exact sequence of $\Lambda$-modules
		coming from the definition of the Selmer groups;
%%%%%%%%%%%%%%%%%%%%%%%%%%%%%%%%%%%%%%%%%%%%%%%%%%%%%%%%%%%%%%%%%%%%%%%%
%%%%		exact sequence coming from defn of Selmer groups		%%%%
%%%%%%%%%%%%%%%%%%%%%%%%%%%%%%%%%%%%%%%%%%%%%%%%%%%%%%%%%%%%%%%%%%%%%%%%
	\begin{eqnarray}\label{exact sequence coming from defn of Selmer groups}
		\bigoplus _{v \in S_{p,F}^{\rm ss}} \left(
			\frac{H^1{(F_{\infty ,v},E[p^{\infty}])}}
			{E^{\pm}(F_{\infty ,v})\otimes \mathbb Q_p/\mathbb Z_p}
			\right)^{\vee}
		& \overset{\iota ^{\pm}}{\longrightarrow} & \Sel (F_{\infty},E[p^{\infty}])^{\vee} \nonumber \\
		& \longrightarrow & \Sel^{\pm} (F_{\infty},E[p^{\infty}])^{\vee}
		\longrightarrow 0.
	\end{eqnarray}

%	By the following two propositions,
%		the proof comes down to the problem on describing the $\Lambda$-module structure of
%		$( H^1(F_{\infty,v},E[p^{\infty}])/
%		E^{\pm}(F_{\infty,v}) \otimes \mathbb Q_p/\mathbb Z_p ) ^{\vee}$.

%%%%%%%%%%%%%%%%%%%%%%%%%%%%%%%%%%%%%%%%%%%%%%%%%%%%%%%%%%%%%%%%%%%%%%%%%%%%%%%%%%%%%%%%
%%%%		[prop]	a property of exact seqn coming from defn of Selmer groups		%%%%
%%%%%%%%%%%%%%%%%%%%%%%%%%%%%%%%%%%%%%%%%%%%%%%%%%%%%%%%%%%%%%%%%%%%%%%%%%%%%%%%%%%%%%%%
	\begin{prop}\label{a property of exact seqn coming from defn of Selmer groups}
		Assume that $\Sel ^{\pm}(F_{\infty}, E[p^{\infty}])^{\vee}$ is $\Lambda$-torsion.
		If each direct summand $( H^1(F_{\infty,v},E[p^{\infty}])/
			(E^{\pm}(F_{\infty,v}) \otimes \mathbb Q_p/\mathbb Z_p) ) ^{\vee}$
			of the leftmost in the exact sequence (\ref{exact sequence coming from defn of Selmer groups})
			is a torsion-free $\Lambda$-module of $\Lambda$-rank $[F_{0,v}:\mathbb Q_p]$,
			then the map $\iota ^{\pm}$ in (\ref{exact sequence coming from defn of Selmer groups})
			is injective.
	\end{prop}
	
	\begin{proof}
		We have
			$\rank _{\Lambda} (\Sel (F_{\infty}, E[p^{\infty}])^{\vee})
			\geq \sum _{v \in S_{p,F}^{\rm ss}}[F_{0,v}:\mathbb Q_p]$
			(cf. \cite{Gre99} Theorem 1.7).
		From (\ref{exact sequence coming from defn of Selmer groups})
			and our assumptions, we see that
		\begin{eqnarray*}
			\sum _{v \in S_{p,F}^{\rm ss}}[F_{0,v}:\mathbb Q_p]
			&=& \rank _{\Lambda} \left(
				\bigoplus _{v \in S_{p,F}^{\rm ss}} \left(
				\frac{H^1(F_{\infty,v},E[p^{\infty}])}{E^{\pm}(F_{\infty ,v})
				\otimes \mathbb Q_p/\mathbb Z_p}
				\right) ^{\vee} \right)\\[1mm]
			& \geq & \rank _{\Lambda} \left(
				\Sel (F_{\infty},E[p^{\infty}])^{\vee}
				\right).
		\end{eqnarray*}
		Thus we get
		\begin{eqnarray*}
			\rank _{\Lambda} \left(
				\bigoplus _{v \in S_{p,F}^{\rm ss}} \left(
				\frac{H^1(F_{\infty,v},E[p^{\infty}])}{E^{\pm}(F_{\infty ,v})
				\otimes \mathbb Q_p/\mathbb Z_p}
				\right) ^{\vee} \right)
				= \rank _{\Lambda} \left(
				\Sel (F_{\infty},E[p^{\infty}])^{\vee}
				\right).
		\end{eqnarray*}
		From this,
			we see that the kernel $\Ker \iota^{\pm}$ is $\Lambda$-torsion.
		Therefore we get the conclusion since the leftmost direct sum
			in the exact sequence (\ref{exact sequence coming from defn of Selmer groups})
			is a torsion-free $\Lambda$-module.
	\end{proof}

%	Our strategy to prove Theorem \ref{main theorem II}
%		is to apply the following proposition to the maps $\iota ^{\pm}$
%		in (\ref{exact sequence coming from defn of Selmer groups}).
	The following proposition is a key tool for the proof of our main theorem.
%%%%%%%%%%%%%%%%%%%%%%%%%%%%%%%%%%%%%%%%%%%%%%%%
%%%%		[prop]	Greenberg's lemma 1		%%%%
%%%%%%%%%%%%%%%%%%%%%%%%%%%%%%%%%%%%%%%%%%%%%%%%
	\begin{prop}[Greenberg \cite{Gre99} p.104--105]\label{Greenberg's lemma 1}
		Let $f:M \rightarrow N$ be an injective homomorphism of $\Lambda$-modules.
		Suppose that $N$ is a finitely generated $\Lambda$-module
		which has no nontrivial finite $\Lambda$-submodule,
		and that $M$ is a free $\Lambda$-module.
		Then the cokernel $\Coker (f)$
		has no nontrivial finite $\Lambda$-submodule.
	\end{prop}
	
	\begin{proof}
		We put $N_1 := \Coker (f)$.
		Then by taking the invariant-coinvariant exact sequence
			of $0 \rightarrow M \rightarrow N \rightarrow N_1 \rightarrow 0$,
			we get an exact sequence
%%%%%%%%%%%%%%%%%%%%%%%%%%%%%%%%%%%%%%%%%%%%%%%%%%%%%%%%%%%%
%%%%		exact sequence in Greenberg's lemma 1		%%%%
%%%%%%%%%%%%%%%%%%%%%%%%%%%%%%%%%%%%%%%%%%%%%%%%%%%%%%%%%%%%
		\begin{eqnarray}\label{exact sequence in Greenberg's lemma 1}
			M^{\Gamma} \longrightarrow
			N^{\Gamma} \longrightarrow
			N_1^{\Gamma} \longrightarrow
			M_{\Gamma}.
		\end{eqnarray}
		Since $M$ is a free $\Lambda$-module, we see that $M^{\Gamma}=0$
			and $M_{\Gamma}$ is a free $\mathbb Z_p$-module
			by Lemma \ref{equivalent conditions on freenes and triviality of finite lambda-submodules}.
		Since $N$ has no non-trivial finite $\Lambda$-submodule,
			we see that $N^{\Gamma}$ is a free $\mathbb Z_p$-module
			by Lemma \ref{equivalent conditions on freenes and triviality of finite lambda-submodules}.
		Thus from the exact sequence (\ref{exact sequence in Greenberg's lemma 1}),
			we see that $N_1^{\Gamma}$ is a free $\mathbb Z_p$-module
			and therefore $N_1$ has no nontrivial finite $\Lambda$-submodule
			again by Lemma \ref{equivalent conditions on freenes and triviality of finite lambda-submodules}.
	\end{proof}

%%%%%%%%%%%%%%%%%%%%%%%%%%%%%%%%%%%%%%%%%%%%
%%%%		[thm]	main theorem II		%%%%
%%%%%%%%%%%%%%%%%%%%%%%%%%%%%%%%%%%%%%%%%%%%
	\begin{thm}\label{main theorem II}
		Assume that both $\Sel ^+(F_{\infty}, E[p^{\infty}])^{\vee}$
			and $\Sel ^-(F_{\infty}, E[p^{\infty}])^{\vee}$
			are $\Lambda$-torsion.
		Then both $\Sel ^+(F_{\infty}, E[p^{\infty}])^{\vee}$
			and $\Sel ^-(F_{\infty}, E[p^{\infty}])^{\vee}$
			have no nontrivial finite $\Lambda$-submodule.
	\end{thm}
	
	\begin{proof}
		By Proposition \ref{Lambda module structure of H^1/E^{pm}}, we have
			\begin{eqnarray*}
				\left(
				\frac{H^1(F_{\infty,v},E[p^{\infty}])}{E^{\pm}(F_{\infty ,v}) \otimes \mathbb Q_p/\mathbb Z_p}
				\right) ^{\vee}
				\cong \Lambda ^{\oplus [F_{0,v}:\mathbb Q_p]}
			\end{eqnarray*}
			for each prime $v \in S_{p,F}^{\rm ss}$.
		Thus we can apply Proposition \ref{a property of exact seqn coming from defn of Selmer groups}
			and Proposition \ref{Greenberg's lemma 1} for $f=\iota^{\pm}$.
		Thus, by Theorem \ref{main theorem I}, we get the desired result.
	\end{proof}
	
	{\bf Acknowledgements}
	
	The authors are grateful to Professor Masato Kurihara
		for suggesting this problem
		and idea on the proof of Theorem \ref{supplementary theorem}
		(=Theorem \ref{supplementary theorem (Intro)}),
		advice, helpful discussions and generous support.
	The authors thank Professor Sin-ichi Kobayashi
		for informative conversations,
		and Professor Robert Pollack for letting us know
		Myoungil Kim's result \cite{MKim11},
		and Florian Sprung for conversation relating to \cite{Kim13} and \cite{Kim14}.
	The authors are partially supported by JSPS Core-to-core program,
		Foundation of a Global Research Cooperative Center in Mathematics focused on
		Number Theory and Geometry.

%\begin{thm}[\cite{Rosen1} Corollary of Theorem 11.6]\label{a}

	%%%%%%%%%%%参考文献%%%%
	

\begin{thebibliography}{ABCD}
	\bibitem[1]{Gre89}
		{R. Greenberg},
		{Iwasawa theory for $p$-adic representations},
		{Advanced Studies in Pure Mathematics {\bf 17} (1989), 97--137}.
	\bibitem[2]{Gre99}
		{R. Greenberg},
		{Iwasawa theory for elliptic curves},
		{in: Arithmetic theory of elliptic curves, Cetraro, Italy 1997,
		Springer Lecture Notes in Math. {\bf 1716} (1999), 51--144}.
	\bibitem[3]{Hachimori-Matsuno00}
		{Y. Hachimori and K. Matsuno},
		{On finite $\Lambda$-submodules of Selmer groups of elliptic curves},
		{Proc. Amer. Math. Soc. {\bf 128} (2000), 2539--2541}.
	\bibitem[4]{Hazewinkel77}
		{M. Hazewinkel},
		{On Norm maps for one dimensional formal groups III},
		{Duke Math. J. {\bf 44} (1977), 305--314}.
	\bibitem[5]{Hon70}
		{T. Honda},
		{On the theory of commutative formal groups},
		{J. Math. Soc. {\bf 22} (1970), 213--246}.
	\bibitem[6]{IP06}
		{A. Iovita and R. Pollack},
		{Iwasawa theory of elliptic curves at supersingular primes over $\mathbb Z_p$-extensions of number fields},
		{J. reine angew. Math. {\bf 598} (2006), 71--103}.
	\bibitem[7]{Kim07}
		{B. D. Kim},
		{The parity conjecture for elliptic curves at supersingular reduction primes},
		{Compositio Math. {\bf 143} (2007), 47--72}.
	\bibitem[8]{Kim13}
		{B. D. Kim},
		{The plus/minus Selmer groups for supersingular primes},
		{J. Aust. Math. Soc. {\bf 95} (2013), 189--200}.
	\bibitem[9]{Kim14}
		{B.D. Kim},
		{Signed-Selmer groups over the maximal $\mathbb Z_p^2$-extension of an imaginary quadratic field},
		{Canad. J. Math. {\bf 66} (2014), 826--843}.
	\bibitem[10]{MKim11}
		{M. Kim},
		{Projectivity and Selmer groups in the non-ordinary case},
		{Dissertation, 2011}.
	\bibitem[11]{Kob03}
		{S. Kobayashi},
		{Iwasawa theory for elliptic curves at supersingular primes},
		{Invent. math. {\bf 152} (2003), 1--36}.
	\bibitem[12]{Kob13}
		{S. Kobayashi},
		{The $p$-adic Gross-Zagier formula for elliptic curves at supersingular primes},
		{Invent. math. {\bf 191} (2013), 527--629}.
	\bibitem[13]{Kurihara02}
		{M. Kurihara},
		{On the Tate Shafarevich groups over cyclotomic fields of an elliptic curve with supersingular reduction I},
		{Invent. math. {\bf 149} (2002), 195--224}.
	\bibitem[14]{Matsuno03}
		{K. Matsuno},
		{Finite $\Lambda$-submodules of Selmer groups of abelian varieties over cyclotomic $\mathbb Z_p$-extensions},
		{J. Number Theory, {\bf 99} (2003), 415--443}.
	\bibitem[15]{NSW}
		{J. Neukirch, A. Schmidt, K. Wingberg},
		{\em Cohomology of Number Fields},
		{Springer-Verlag, 2008}.
	\bibitem[16]{Sprung12}
		{F. Sprung},
		{Iwasawa theory for elliptic curves at supersingular primes: A pair of main conjectures},
		{Journal of Number Theory {\bf 132} (2012) 1483--1506}.
	\bibitem[17]{Takeji14}
		{S. Takeji},
		{On the Selmer modules of an elliptic curve with supersingular reduction (Japanese)},
		{Master's Thesis, 2014}.
	\end{thebibliography}
\end{document}